\documentclass[a4paper,11pt, reqno]{amsart}
\usepackage[margin=2cm]{geometry}

\usepackage{natbib}

\usepackage{amssymb,amsmath,amsthm}
\usepackage{color}
\usepackage[table]{xcolor}
\usepackage{adjustbox}
\usepackage{graphicx}
\usepackage{subcaption}%subfigure

\usepackage{float}
\usepackage{hyperref}
\usepackage{multirow}
\usepackage{multicol}

\usepackage{natbib}

\usepackage{hyperref}
\usepackage{amsmath,amsfonts}
\usepackage{amssymb}
\usepackage{url}
\usepackage{adjustbox}
\usepackage{multirow}

\usepackage{lineno}
\usepackage[linesnumbered,lined,boxed,commentsnumbered]{algorithm2e}

\usepackage{subcaption}
\usepackage{adjustbox}

 \usepackage{hyperref}
\usepackage{amsthm}

\def\R{\mathbb{R}}

\def\M{\mathcal{M}}

\def\R{\mathbb{R}}

\def\M{\mathcal{M}}
\def\it{\texttt{it}}

\def\x{\bold{x}}

\usepackage{multirow}
\newcommand{\dsum}{\displaystyle\sum}

\newtheorem{theorem}{Theorem}
\newtheorem{lemma}[theorem]{Lemma} 
\newtheorem{definition}{Definition}[section]

\usepackage{tikz}
\usetikzlibrary{positioning,calc}
\usetikzlibrary{arrows.meta, decorations.pathreplacing}

\let\origmaketitle\maketitle
\def\maketitle{
	\begingroup
	\def\uppercasenonmath##1{} % this disables uppercasing title
	\let\MakeUppercase\relax % this disables uppercasing authors
	\origmaketitle
	\endgroup
}

\begin{document}

\title[]{\Large Identifying Self-Amplifying Hypergraph Structures through Mathematical Optimization}

\author[V. Blanco, G. Gonz\'alez, \MakeLowercase{and} P. Gagrani]{
{\large V\'ictor Blanco$^{\dagger}$, Gabriel Gonz\'alez$^{\dagger}$, and Praful Gagrani$^{\dagger}$}\medskip\\
$^\dagger$Institute of Mathematics (IMAG), Universidad de Granada\\
$^\star$Institute of Industrial Science, University of Tokyo\\
\texttt{vblanco@ugr.es}, \texttt{ggdominguez@ugr.es}, \texttt{praful@sat.t.u-tokyo.ac.jp}
}

\maketitle

\begin{abstract}
In this paper, we introduce the concept of \emph{self-amplifying structures} for hypergraphs, positioning it as a key element for understanding propagation and internal reinforcement in complex systems. To quantify this phenomenon, we define the maximal amplification factor, a metric that captures how effectively a subhypergraph contributes to its own amplification. We then develop an optimization-based methodology to compute this measure. Building on this foundation, we tackle the problem of identifying the subhypergraph maximizing the amplification factor, formulating it as a mixed-integer nonlinear programming (MINLP) problem. To solve it efficiently, we propose an exact iterative algorithm with proven convergence guarantees. In addition, we report the results of extensive computational experiments on realistic synthetic instances, demonstrating both the relevance and effectiveness of the proposed approach. Finally, we present a case study on chemical reaction networks—including the Formose reaction and E. coli core metabolism—where our framework successfully identifies known and novel autocatalytic subnetworks, highlighting its practical relevance to systems chemistry and biology.
\end{abstract}
 

\section{Introduction \label{sec: introduccion}}

Networks are combinatorial structures composed of nodes and arcs (or edges) that represent the interactions between entities. Typically, nodes model the components of a system, while arcs capture their relationships. Modeling systems as networks has proven essential for uncovering structural and dynamical properties across a wide range of domains. Over time, networks have become a central object of study in both theoretical and applied mathematics, with significant advances from multiple perspectives, including algebraic spectral graph theory~\citep{chung1997spectral}, geometric graph embeddings~\citep{heinonen2003geometric}, combinatorial graph theory~\citep{harris2008combinatorics}, arithmetic graph theory~\citep{acharya1990arithmetic}, and combinatorial optimization~\citep{cook1994combinatorial}. This extensive toolkit has rendered graph theory indispensable for modeling, analyzing, and optimizing complex systems.

However, a key limitation of classical network models is their inherently pairwise nature: each arc connects exactly two nodes. While this suffices for many applications—such as social networks, where an arc may indicate a "follows" relationship between users, or logistics networks, where arcs represent transportation routes—many real-world interactions involve more complex, multi-way relationships.

Hypergraphs provide a natural generalization of networks by allowing edges (hyperarcs) to connect multiple nodes simultaneously. This extended framework has recently proven crucial in modeling higher-order interactions across disciplines. For example, in biology and chemistry, reactions often involve multiple reactants and products~\citep{andersen_chemical_2019}; in manufacturing and process engineering, multiple input materials may be transformed into several outputs~\citep{nagy2022hypergraph}; and in multi-agent systems, tasks may require coordinated actions from several agents to produce multiple outcomes~\citep{wang2024hyper}. A sophisticated application of hypergraph encoding includes also complex model representation~\cite{michelena1997hypergraph}.  Despite their expressiveness, hypergraphs are still comparatively underexplored—particularly when integrated into combinatorial optimization models. Moreover, while many results in graph theory extend naturally to hypergraphs, others break down completely. For instance, the minimum edge cover problem, solvable in polynomial time for graphs~\citep{garey1979computers}, becomes NP-hard in hypergraphs~\citep{repisky2007covering}. Similar complexity jumps occur for maximum-cardinality matching~\citep{berman2003improved}, the minimum vertex cover problem~\citep{guruswami2010inapproximability}, and even shortest path computations in directed hypergraphs~\citep{ausiello2017directed}. The situation becomes even more intricate in directed or multiset hypergraphs, where many foundational questions remain open. Thus, while hypergraphs offer greater modeling power, they also present new challenges and opportunities in the development of combinatorial optimization tools.

In this work, we investigate a combinatorial optimization problem on hypergraphs that arises in the analysis of different types of  propagation processes in dynamical systems~\cite{carletti2020dynamical,ferraz2021phase}. In large hypergraphs—used to model temporal evolution—understanding system dynamics involves identifying the key elements and interactions that drive change. This task is particularly challenging in real-world networks due to their complexity and scale. For instance, in biochemical networks, detecting the specific motifs responsible for evolutionary transitions can help answer questions related to the \emph{Origin of Life}—that is, which molecules, and under what conditions, catalyzed the emergence of life~\citep{hordijk2018autocatalytic,hordijk2004detecting,hordijk_autocatalytic_2018}. Similarly, in social networks, identifying the most influential users and their interactions helps advertisers effectively target and amplify their campaigns. In manufacturing, understanding how certain inputs contribute disproportionately to output can guide resource allocation and pricing strategies.

We introduce the notion of a \emph{self-amplifying subhypergraph} in the context of multi-directed hypergraphs. A self-amplifying subhypergraph comprises a subset of nodes and hyperarcs from a larger hypergraph, along with a distinguished set of nodes that appear as inputs, capturing the idea of \emph{self}. This substructure satisfies an amplification property: through appropriate combinations of the included hyperarcs, the net production of the distinguished nodes exceeds their net consumption. This captures the essence of a subsystem capable of reinforcing or amplifying its population via internal interactions. While similar notions have appeared in domains such as chemical reaction networks (e.g., autocatalytic sets~\cite{blokhuis_universal_2020}) and production economics (e.g., Von Neumann growth models~\cite{de2012neumann}), our formulation generalizes these ideas to arbitrary multi-directed hypergraphs.

Enumerating all such substructures is computationally intractable. As shown by \citet{gagrani2023geometry}, even the number of inclusion-minimal self-amplifying subhypergraphs can grow exponentially with the number of nodes. To overcome this, we propose a quantitative measure inspired by \citet{neumann1945model} to evaluate the amplification capacity of a subhypergraph. Specifically, we define the \emph{amplification factor} of a subhypergraph, given a set of intensity weights on the hyperedges, as the minimum ratio—across all involved nodes—between the total output and the total input of each node under those intensities. We then define the \emph{maximal amplification factor} (MAF) as the highest such ratio attainable by optimizing the edge weights. While this index can be computed for any subhypergraph, we establish a tight connection between the MAF and the notion of self-amplification: we derive necessary and sufficient conditions under which a subhypergraph is self-amplifying based on its MAF. This provides a practical method for identifying self-reinforcing subsystems without exhaustive enumeration. The main distinction between the MAF and other hypergraph-based measures used across different fields is that the MAF quantifies a hypergraph's capacity to expand or grow over time. In other words, MAFs characterize the system's asymptotic behavior, specifically, its potential maximal growth.

To compute the MAF, we develop a nonlinear optimization model that falls into the class of generalized fractional programs~\citep{barros1996new,blanco2013minimizing,crouzeix1991algorithms}, which are closely related to the well-known generalized eigenvalue problem~\cite{nishioka2025minimization}. These problems are typically nonconvex and computationally demanding. To address this, we design an exact iterative algorithm with guaranteed convergence properties.

We then extend our analysis to a higher-level combinatorial problem: identifying, within a given hypergraph, the subhypergraph that maximizes the amplification factor. Solving this problem enables the discovery of the most influential substructures—those most capable of driving system-wide propagation or growth. We formulate this as a mixed-integer nonlinear programming (MINLP) problem, using binary decision variables to encode the selection of nodes and arcs. This model builds on the MAF computation framework, with additional constraints to manage the combinatorial complexity. Furthermore, we introduce application-driven structural constraints to guide the search toward interpretable and practically relevant substructures.

The remainder of this paper is organized as follows. In Section~\ref{sec:preliminaries}, we introduce the necessary notation and formally define the MAF in multi-directed hypergraphs. Section~\ref{sec:computing_MAF} presents the optimization framework for computing the MAF of a fixed subhypergraph. In Section~\ref{sec:optimal_subhypergraphs}, we extend the model to identify optimal subhypergraphs. In Section~\ref{sec:features} we add some features to our models that allow us to add specific constraints tailored to particular scenarios. Section~\ref{sec:experiments} reports computational results on synthetic datasets simulating chemical reaction networks, illustrating the effectiveness and limitations of our approach. Section~\ref{sec:case_study} presents real-world case studies involving \emph{Chemical Reaction Networks} (CRNs). Finally, Section~\ref{sec:conclusions} summarizes our contributions and outlines directions for future research.

\section{Preliminaries \label{sec:preliminaries}}

In this section, we introduce the notation and core definitions used throughout the paper. The section is organized in two parts. First, apart from providing the main definitions required for this work, we introduce the notion of self-amplifying subhypergraph. In the second part, we define the amplification factor and analyze subhypergraphs that maximize this metric and its relation with self-amplifying hypergraphs.

\subsection{Self-amplifying subhypergraphs}

The primary mathematical object underlying our work is the directed hypergraph, which generalizes the concept of (di)graphs by allowing hyperedges to connect multiple nodes in a directed fashion. Unlike standard graphs, where edges link pairs of nodes, directed hypergraphs enable richer modeling of complex relationships by connecting sets of nodes to other sets of nodes. While various definitions of directed hypergraphs exist in the literature, we adopt the most commonly used formulation.

\begin{definition} A directed hypergraph is a tuple $\mathcal{H} = (\mathcal{N}, \mathcal{A})$, where $\mathcal{N}$ is the set of nodes and $\mathcal{A}$ is the set of hyperarcs, where each hyperarc $a \in \mathcal{A}$ is represented as an ordered pair $a = (S_a, T_a)$, with $\emptyset \neq S_a, T_a \subseteq \mathcal{N}$, the \textbf{source} and \textbf{target} sets of hyperarc $a$, respectively. 
\end{definition}

Observe that when both the source and target sets of every hyperarc contain exactly one node (i.e., $|S_a| = |T_a| = 1$ for all $a \in \mathcal{A}$), the structure reduces to a standard directed graph. However, directed hypergraphs offer a more flexible modeling framework, particularly suitable for representing systems involving higher-order dependencies that cannot be captured by digraphs alone. A prominent example arises in production planning, where technologies (represented by hyperarcs) may require multiple input materials and generate multiple output products, often with shared inputs or outputs across technologies. In this context, $S_a$ denotes the set of input nodes, and $T_a$ the set of output nodes associated with hyperarc $a$.

A further generalization is the multihypergraph, which allows multiple hyperarcs to exist between the same source and target sets. This multiplicity enables modeling scenarios where several instances of the same transformation occur, or where a hyperarc consumes or produces multiple units of a given product. Such modeling is essential in input-output production systems, where technologies may consume or generate multiple units of specific goods. While multihypergraphs can theoretically be simplified into standard hypergraphs, this often results in the loss of valuable structural information—information that can be crucial for analysis and decision-making~\citep{klamt2009hypergraphs,battiston2021physics,bick2023higher}.

Next, we define the input and output incidence matrices of a hypergraph:

\begin{definition}
Let $\mathbb{S}$ (respectively, $\mathbb{T}$) denote the \textbf{input} (respectively, \textbf{output}) incidence matrix of the hypergraph $\mathcal{H} = (\mathcal{N}, \mathcal{A})$. For a hyperarc $a \in \mathcal{A}$ and a node $v \in {S}_a$ (respectively, $v \in {T}_a$), the entry $\mathbb{S}_{va}$ (respectively, $\mathbb{T}_{va}$) corresponds to the multiplicity of $v$ in the multiset ${S}_a$ (respectively, ${T}_a$), that is, $\mathbb{S}_{va} = |v|$ in ${S}_a$ (respectively, $\mathbb{T}_{va} = |v|$ in ${T}_a$).
\end{definition}

Figure \ref{fig:bipartite} illustrates a simple multi-directed hypergraph represented as a bipartite graph, where circular nodes correspond to hypergraph nodes $\mathcal{N} = \{A, B, C\}$ and square nodes represent hyperarcs $a_1 = (\{C\}, \{A, A, A\})$, $a_2 = (\{A, A, A\}, \{A, B\})$ and $a_3 = (\{A, B\}, \{C\})$. The corresponding input and output incidence matrices are:

$$
\mathbb{S} = \begin{pmatrix}
    3 & 1 & 0\\
    0 & 1 & 0\\
    0 & 0 & 1
\end{pmatrix}, \quad \mathbb{T} = \begin{pmatrix}
    1 & 0 & 3\\
    1 & 0 & 0\\
    0 & 1 & 0
\end{pmatrix}.
$$

\begin{figure}[h!] 
\begin{center}
\includegraphics[trim={1.5cm 1cm 1cm 0.9cm},clip,width=0.8\textwidth]{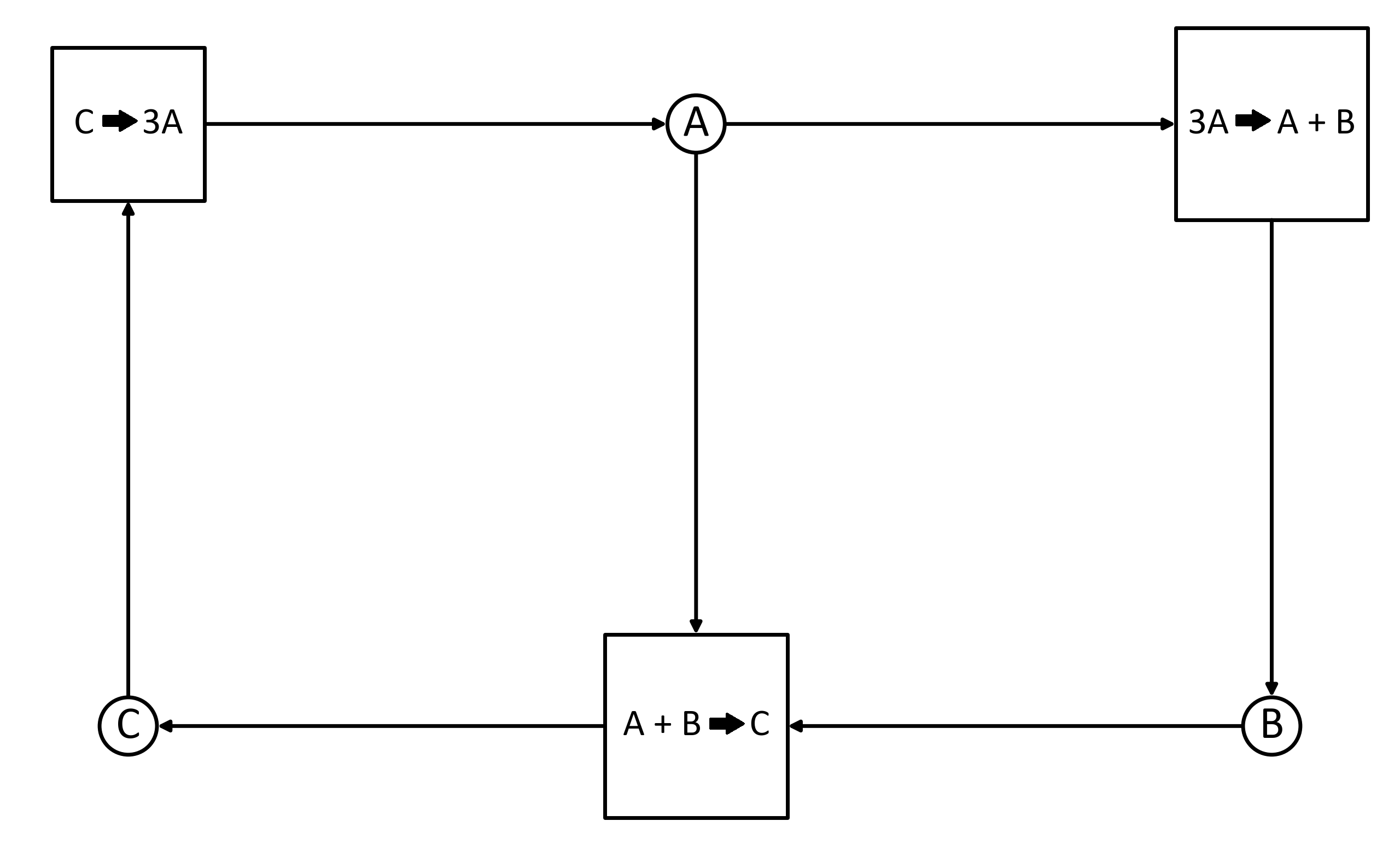}
\end{center}
\caption{ A hypergraph represented as a bipartite graph.\label{fig:bipartite}}
\end{figure}

\begin{definition}
The \textbf{net incidence matrix} of the hypergraph is defined as the difference between the output and input matrices: $\mathbb{Q} = \mathbb{T} - \mathbb{S}$.
\end{definition}

In the example above, the net matrix is:
$$
\mathbb{Q} = \begin{pmatrix}
    -2 & -1 & 3\\
    1 & -1 & 0\\
    0 & 1 & -1
\end{pmatrix}.
$$
Negative entries in $\mathbb{Q}$ indicate nodes that act as sources, while positive entries denote nodes that act as targets. If a node appears in both the source and target of a hyperarc with the same multiplicity, the corresponding entry in $\mathbb{Q}$ is zero, reflecting a net-neutral contribution. More on visualization and the graph-theoretical interpretation of such hypergraphs can be found in~\citep{kraakman2024configuration}.

Each nonzero nonnegative vector $\x \in \mathbb{R}_+^{\mathcal{A}} \setminus \{\mathbf{0}\}$ represents an \emph{intensity vector} that linearly combines the hyperarcs. The vectors $\mathbb{S} \x$ and $\mathbb{T} \x$ represent the input and output incidence of nodes under this combination. The net incidence is given by:
$\mathbb{Q} \x = (\mathbb{T} - \mathbb{S}) \x$,
which indicates the overall balance for each node. A negative value for a node indicates it is being depleted (a \emph{drain node}), while a positive value indicates it is being produced more than consumed (an \emph{amplifying node}).

In the example from Figure~\ref{fig:bipartite}, for the intensity vector $\x = (1, 2, 1)$, we obtain:
$\mathbb{Q} \x = (3, -1, 1)^\top.$
Thus, node $B$ acts as a drain, while $A$ and $C$ are amplifying.

\begin{definition}
A hypergraph $\mathcal{H}' = (\mathcal{N}', \mathcal{A}')$ is called a \emph{subhypergraph} of $\mathcal{H} = (\mathcal{N}, \mathcal{A})$, denoted by $\mathcal{H}' \subseteq \mathcal{H}$, if $\mathcal{A}' \subseteq \mathcal{A}$ and all nodes appearing in the input or output of hyperarcs in $\mathcal{A}'$ belong to $\mathcal{N}'$, i.e., for every $a = (\mathcal{S}_a, \mathcal{T}_a) \in \mathcal{A}'$, we have $\mathcal{S}_a, \mathcal{T}_a \subseteq \mathcal{N}'$.
\end{definition}

Given a subhypergraph $\mathcal{H}' \subseteq \mathcal{H}$ and a subset $\mathcal{M} \subseteq \mathcal{N}'$, the subhypergraph of $\mathcal{H}'$ induced by the nodes in $\mathcal{M}$ is called the \emph{$\mathcal{M}$-restricted subhypergraph} and is denoted by $\mathcal{H}'|_{\mathcal{M}}$. A central concept in this work is that of a self-amplifying subhypergraph, defined below.

\begin{definition}[Self-Amplifying Subhypergraph]\label{def:selfamplifying}
Let $\mathcal{H}$ be a hypergraph and $\mathcal{H}' = (\mathcal{N}', \mathcal{A}') \subseteq \mathcal{H}$. We say that $\mathcal{H}'$ is \textbf{self-amplifying} with respect to a node set $\mathcal{M} \subseteq \mathcal{N}'$ if the following conditions are satisfied:

\begin{enumerate}
\item \textbf{Self-sufficiency:} Each hyperarc in $\mathcal{A}'$ must have at least one input and one output node in $\mathcal{M}$, and every node in $\mathcal{M}$ must appear at least once in the source and once in the target of some hyperarc in $\mathcal{A}'$:
\begin{align}
&\forall a \in \mathcal{A}',\ \exists v, v' \in \mathcal{M}: \quad \mathbb{T}_{va} > 0,\quad \mathbb{S}_{v'a} > 0,\tag{$A^a$} \label{autonomous1}\\
&\forall v \in \mathcal{M},\ \exists a, a' \in \mathcal{A}': \quad \mathbb{T}_{va} > 0,\quad \mathbb{S}_{va'} > 0.\tag{$A^v$} \label{autonomous2}
\end{align}

\item \textbf{Net-positive realizability:} There exists a strictly positive intensity vector $\x \in \mathbb{R}_{>0}^{\mathcal{A}'}$ such that the net incidence for every node in $\mathcal{M}$ is strictly positive:
\begin{align}
\sum_{a \in \mathcal{A}'} \mathbb{Q}_{va} x_a > 0,\quad \forall v \in \mathcal{M}.
\end{align}
\end{enumerate}
\end{definition}

These notions are closely related to autonomy and productivity concepts studied in the context of Chemical Reaction Networks~\citep{blokhuis_universal_2020,gagrani2023geometry}.

For a self-amplifying subhypergraph over $\mathcal{M}$, the set $\mathcal{M}$ is referred to as the set of \emph{self-amplifying nodes}. The $\mathcal{M}$-restricted subhypergraph is said to be \textbf{minimally self-amplifying} if it does not contain a proper self-amplifying subhypergraph. The matrix $\mathbb{Q}|_{\mathcal{M}, \cdot}$ associated with such a minimal subhypergraph is called an \emph{amplification core}.

The example from Figure \ref{fig:bipartite} is minimally self-amplifying since:
\begin{enumerate}
    \item It is a self-sufficient hypergraph, as every node appears in both the source and target sets of at least one hyperarc.
    \item It is also net-positive realizable, as there exists a intensity vector $\x = (0, 0, 1)^\top$ for which the $\mathbb{Q} \x = (3, 0, -1)^\top$ and $\sum_{a \in \mathcal{A}} \mathbb{Q}_{va}x_a = 2 > 0.$
    \item It does not contain any other self-amplifying subhypergraph.
\end{enumerate}

\subsection{Maximal Amplification Factor}

Multi-directed subhypergraphs can simulate dynamical systems that evolve over discrete time. Let $\mathcal{H}$ be a hypergraph, and let $\mathcal{H}' = (\mathcal{N}', \mathcal{A}')$ be a subhypergraph that evolves through a sequence of time-indexed intensity vectors $\x^t \geq 0$, for $t \in \mathbb{Z}_+$, applied to the hyperarcs in $\mathcal{H}'$.

Suppose that the nodes in a subset $\mathcal{M} \subseteq \mathcal{N}'$ are positively amplified over time. This means that, for every $v \in \mathcal{M}$ and every $a \in \mathcal{A}'$, we have:
\begin{align*}
    \mathbb{T}_{va} \x_a^t \geq \mathbb{S}_{va} \x_a^t.
\end{align*}
where here $\geq$ indicates the componentwise order in $\mathbb{R}^{|\mathcal{M}|}$.

Let $\mathbb{S}_{\mathcal{M}}$ and $\mathbb{T}_{\mathcal{M}}$ denote the submatrices of $\mathbb{S}$ and $\mathbb{T}$ (row-restricted to the nodes in $\mathcal{M}$). Considering that the output of nodes at time $t$ becomes the input at time $t+1$, the following inequality holds:
\begin{equation} \label{gf0}
    \mathbb{T}_{\mathcal{M}} \x^t \geq \mathbb{S}_{\mathcal{M}} \x^{t+1}.
\end{equation}

If we further assume that the intensity at time $t+1$ is a scalar multiple of the intensity at time $t$, i.e., $\x^{t+1} = \alpha \x^t$, thereby preserving the system's structure, we obtain:
\begin{align*}
    \mathbb{T}_{\mathcal{M}} \x^t \geq \alpha \mathbb{S}_{\mathcal{M}} \x^t.
\end{align*}
We then say that the subhypergraph $\mathcal{H}'$ with initial intensity $\x^0$ \emph{amplifies} if $\alpha \geq 1$, and \emph{contracts} if $\alpha < 1$.

The focus of this paper is the analysis of how to determine initial intensity vectors that induce maximal amplification within the system. This leads to the following key definition:

\begin{definition}[Amplification Factor] \label{def:MGF}
Let $\mathcal{H} = (\mathcal{N}, \mathcal{A})$ be a hypergraph, let $\mathcal{H}' = (\mathcal{N}', \mathcal{A}') \subseteq \mathcal{H}$ be a subhypergraph, and let $\emptyset \neq \mathcal{M} \subset \mathcal{N}'$. Given an intensity vector $\x \in \mathbb{R}_{>0}^{\mathcal{A}}$, the \emph{amplification factor} of $\mathcal{H}'$ with respect to $\x$ is defined as:
\begin{align*}
    \alpha(\mathcal{H}', \mathcal{M}; \x) = \inf_{v \in \mathcal{M}} \frac{\dsum_{a \in \mathcal{A}'} \mathbb{T}_{va} \x_a}{\dsum_{a \in \mathcal{A}'} \mathbb{S}_{va} \x_a}.
\end{align*}
This value represents the smallest amplification rate achieved by any node in $\mathcal{M}$ under the intensity vector $\x$. Since $\x \neq\mathbf{0}$, and $\mathbb{S}$ is a nonzero, nonnegative integer matrix, the infimum above can be replaced by a minimum.

The \textbf{Maximal Amplification Factor} (MAF) of $\mathcal{H}'$ over $\mathcal{M}$ is the maximum amplification factor attainable over all strictly positive flows $\x \in \mathbb{R}_{>0}^{\mathcal{A}}$:
\begin{align*}
    \alpha(\mathcal{H}', \mathcal{M}) = \max_{\x \in \mathbb{R}^{\mathcal{A}}_{>0}} \alpha(\mathcal{H}', \mathcal{M}; \x).
\end{align*}
\end{definition}

A key contribution of this paper is the development of computational strategies for analyzing expanding hypergraphs using this metric. The following result establishes a foundational link between the amplification factor and the concept of self-amplifying subhypergraphs (Definition \ref{def:selfamplifying}). In particular, we show that the amplification factor provides a certificate for identifying self-amplifying structures.

\begin{lemma} \label{lemma:1}
If $\mathcal{H}'$ is self-sufficient in $\mathcal{M}$, then $0 < \alpha(\mathcal{H}', \mathcal{M}) < \infty$.
\end{lemma}
\begin{proof}
Self-sufficiency of $\mathcal{H}'$ in $\mathcal{M}$ implies that each node in $\mathcal{M}$ is both produced and consumed by at least one hyperarc in $\mathcal{A}'$. Under these conditions, \citet[Theorem 9.8]{gale1989theory} guarantees the existence of a strictly positive amplification factor $\alpha$. Since every hyperarc consumes at least one node, the denominator $\sum_{a \in \mathcal{A}'} \mathbb{S}_{va} \x_a > 0$ for any intensity vector $\x$, ensuring that $\alpha(\mathcal{H}', \mathcal{M}) < \infty$. Similarly, because each node is an output of at least one hyperarc, the numerator $\sum_{a \in \mathcal{A}'} \mathbb{T}_{va} \x_a > 0$, guaranteeing $\alpha(\mathcal{H}', \mathcal{M}) > 0$.
\end{proof}

\begin{lemma} \label{lemma:2}
The subhypergraph $\mathcal{H}'$ is net-positive realizable in $\mathcal{M}$ if and only if $\alpha(\mathcal{H}', \mathcal{M}) > 1$.
\end{lemma}
\begin{proof}
The result follows directly from the definition. If $\mathcal{H}'$ is net-positive realizable in $\mathcal{M}$, then there exists an intensity vector $\x$ such that $\sum_{a \in \mathcal{A}'} \mathbb{T}_{va} \x_a > \sum_{a \in \mathcal{A}'} \mathbb{S}_{va} \x_a$ for all $v \in \mathcal{M}$, which implies $\alpha(\mathcal{H}', \mathcal{M}; \x) > 1$. Conversely, if $\alpha(\mathcal{H}', \mathcal{M}) > 1$, then such an $\x$ exists and the inequality is satisfied.
\end{proof}
\begin{lemma} \label{lemma:autonomous}
   If $0<\alpha(\mathcal{H}^{'},\mathcal{M}) < \infty$, then $\mathcal{H}^{'}$ is entity-self-sufficient.  
\end{lemma}
\begin{proof}
Since $\mathbb{S}$ is nonnegative, and for each $v \in \mathcal{M}$ we have that $\dsum_{a \in \mathcal{A}^{'}} \mathbb{T}_{va} \x_a > 0$ and $\dsum_{a \in \mathcal{A}^{'}} \mathbb{S}_{va} \x_a > 0$, there must exist at least $a, a'\in \mathcal{A}^{'}$ such that $\mathbb{S}_{va}>0$ and $\mathbb{T}_{va^{'}}>0$ (entity self-sufficiency).

On the other hand, for all $a \in \mathcal{A}^{'}$ such that $\mathbb{T}_{va} = 0$ for all $v \in \M$, due to the maximization criterion, the optimal value would always be obtained when $x_a = 0$ for these hyperarcs (which is equivalent to excluding them from the subnetwork).

Let $Q = \{a \in \mathcal{A}^{'}: \mathcal{S}_{va}=0, \forall v \in \mathcal{M} \}$, then the expression for the amplification factor for a given $\x$ is:
\begin{equation*}
\inf_{v \in \M} \frac{\dsum_{a' \in \mathcal{A}^{'}} \mathbb{T}_{va^{'}} \x_{a^{'}}}{\dsum_{a' \in \mathcal{A}^{'}} \mathbb{S}_{va^{'}} \x_{a^{'}}} = \inf_{v \in \M} \left(\frac{\dsum_{a' \in \mathcal{A}^{'}\backslash Q} \mathbb{T}_{va^{'}} \x_{a^{'}}}{\dsum_{a' \in \mathcal{A}^{'}\backslash Q} \mathbb{S}_{va^{'}} \x_{a^{'}}} +  \frac{\dsum_{a\in Q}\mathbb{T}_{va} \x_{a}}{\dsum_{a^{'} \in \mathcal{A}^{'}\backslash Q} \mathbb{S}_{va^{'}} \x_{a^{'}}}\right).
\end{equation*}
Since the term $\dsum_{a \in Q}\mathbb{T}_{va} \x_{a} \geq 0$ only appears in the numerator of the expression, when maximizing over $\x$, the solution is reached for values $x_{a}=\infty$ for $a \in Q$, leading to $\alpha(\mathcal{H}^{'},\mathcal{M}) = \infty$, contradicting the assumption that the amplification factor is finite.
\end{proof}

Combining the above results, we arrive at the central connection between the concepts introduced in this work.

\begin{theorem} \label{th:1}
Let $\mathcal{H} = (\mathcal{N}, \mathcal{A})$ be a hypergraph, $\mathcal{H}' = (\mathcal{N}', \mathcal{A}') \subseteq \mathcal{H}$ a subhypergraph, and $\emptyset \neq \mathcal{M} \subset \mathcal{N}'$. Then:

\begin{enumerate}
    \item If $\mathcal{H}'$ is self-amplifying over $\mathcal{M}$, then $1 < \alpha(\mathcal{H}', \mathcal{M}) < \infty$.
    \item Conversely, if $\mathcal{H}'$ is self-sufficient (i.e., $0 < \alpha(\mathcal{H}', \mathcal{M}) < \infty$) and net-positive realizable (i.e., $\alpha(\mathcal{H}', \mathcal{M}) > 1$), then $\mathcal{H}'$ is self-amplifying in $\mathcal{M}$.
\end{enumerate}
\end{theorem}

\section{Computing the Maximal Amplification Factor} \label{sec:computing_MAF}

In this section, we describe an algorithmic approach to exactly compute the MAF of a given subhypergraph. 

Note that this rate cannot be directly computed, as it depends on identifying the intensity vector that achieves the maximum amplification factor—a nontrivial task that requires solving an optimization problem.
Here, we provide an optimization-based framework for its computation.

The MAF of a subnetwork $\mathcal{H}^{'}$ can be formulated as the following mathematical optimization problem:
\begin{align}
    \alpha(\mathcal{H}^{'},\mathcal{M}): = \max & \;\alpha\nonumber\\
    \mbox{s.t. } & \mathbb{T}_\M\x \geq \alpha \mathbb{S}_\M \x,\tag{MAF}\label{MGF}\\
    & \mathbb{S}_\M \x \geq 1, \nonumber\\  
    & \x \in \R^{\mathcal{A}^{'}}_+, \alpha \geq 0.\nonumber
\end{align}

In the above formulation $\alpha$ represents the MAF, and $\x$ indicates a intensity vector. The objective function is the amplification factor, which is adequately defined by the two sets of constraints. Observe that the constraints $\mathbb{T}_\mathcal{M} \geq \alpha \mathbb{S}_\mathcal{M} \x$ enforce $\alpha$ to be properly defined. The constraint $\mathbb{S}_\mathcal{M} \x \geq 1$, together with $\x \in \mathbb{R}^{|\mathcal{A}'|}_+$ is equivalent to assuring that $\x \neq \mathbf{0}$.

Note that the problem above can be interpreted as selecting intensities that maximize the worst-case (i.e., minimum) amplification factor among the nodes of the subhypergraph.

The problem \eqref{MGF} is neither convex nor concave, and thus, standard convex optimization tools are not directly applicable. However, the MAF can also be formulated as a generalized fractional programming problem, which is computationally challenging (e.g., \cite{barros1996new, blanco2013minimizing, crouzeix1991algorithms}). To solve \eqref{MGF}, we adapt a Dinkelbach-type procedure (e.g., \cite{crouzeix1991algorithms}), which iteratively computes $\alpha(\mathcal{H}^{'}, \mathcal{M})$ and guarantees convergence to the actual value.

In what follows, we analyze the computation of the MAF for a given sub-hypergraph, $\mathcal{H}^{'}=(\mathcal{N}^{'}, \mathcal{A}^{'})$, and a given subset $\mathcal{M} \subseteq \mathcal{N}^{'}$. A specific instance of this occurs when $\mathcal{N}^{'} = \mathcal{M}$, meaning all nodes in the sub-hypergraph are amplification factor-inducing nodes.

The algorithm that we propose for computing the MAF is based on solving exactly the mathematical optimization model in \eqref{MGF} by applying an ad-hoc Dinkelbach-type procedure. Specifically, we begin by initializing the intensity vector to $\mathbf{x}_0$. For each intensity vector obtained during the procedure, we compute the amplification factor induced by that intensity, denoted as $\alpha_{\text{it}}$. Next, we solve the following auxiliary linear program to compute the optimal values $\rho$ and $\mathbf{x}_{\text{it}}$:
\begin{align*}
(\rho^*, \bar \x_{it}) \in \arg\max & \,\,\rho\\
\mbox{s.t.} &\,\, \rho \leq  \mathbb{T}_\M \x - \alpha_{\it} \mathbb{S}_\M \x, &&\\
& \mathbb{S}_\M \x \geq \mathbf{1}, &&\\
& \x \in \R_+^{|\mathcal{A}^{'}|}, \rho \in \R_+.
\end{align*}
In case $\rho^* = 0$, the obtained solution, $(\alpha_{\texttt{it}}, \x)$ is optimal, and the amplification factor for this intensity vector, $\alpha_{\text{it}}$, corresponds to the MAF. Otherwise, the intensity vector is updated based on the solution of the linear program, $\bar \x_{\texttt{it}}$, and the process is repeated. 

\begin{algorithm}[h]
\caption{Computation of the amplification factor of subnetwork $\mathcal{H}'$ with a given set of self-amplifying entities $\M$.\label{alg:1}}
\SetKwInOut{Input}{input}\SetKwInOut{Output}{output}

\Input{$\mathcal{H}^{'} = (\mathcal{N}^{'}, \mathcal{A}^{'})$, $\mathcal{M} \subset \mathcal{N}^{'}$, $\it =0$, and $\x_0 \in \R_+^{|\mathcal{A}^{'}|}$.}
\Output{The amplifying factor $\alpha(\mathcal{H}^{'},\mathcal{M})$}
Stop=False \\
\While{Stop=False}{
Define $\alpha^t_{\it} =  \alpha(\x_{it})$ and 
solve the linear program:
\begin{align*}
(\rho, \bar \x_{it}) \in \arg\max & \,\,\rho\\
\mbox{s.t.} &\,\, \rho \leq  \mathbb{T}_\M \x - \alpha_{\it} \mathbb{S}_\M \x, &&\\
& \mathbb{S}_\M \x \geq \mathbf{1}, &&\\
& \x \in \R_+^{|\mathcal{A}^{'}|}, \rho \in \R_+.
\end{align*}
\eIf{$\rho=0$}{
Stop=True, \\
$\alpha(\mathcal{H}^{'},\mathcal{M}) = \alpha_\it$}{
$\x_{\it+1} = \bar \x_{\it}$,\\
$\it=\it+1$}}
\end{algorithm}

A pseudocode for our proposed algorithm is detailed in Algorithm \ref{alg:1}.

 The convergence of this procedure described in Algorithm \ref{alg:1} is guaranteed by \cite{crouzeix1985algorithm}.
\begin{theorem}
    If the feasible intensity vectors $\x$ are upper bounded, Algorithm \ref{alg:1} converges linearly (in terms of the number of iterations) to the optimal solution.
\end{theorem}
\begin{proof}
    The proof follows by applying the results in \cite{crouzeix1985algorithm}, and showing that the MAF algorithm we propose satisfies the hypotheses for their application. Specifically, it is not difficult to verify that the feasible regions of the problems under study are compact and that the functions involved in the quotients defining the MAF are linear (and hence continuous).
\end{proof}

\section{Optimal subhypergraphs under maximal amplyfing factors} \label{sec:optimal_subhypergraphs}

The main methodological contribution of this paper is the mathematical model we propose to construct self-sufficient subnetworks that maximize the amplification factor. Note that, if the set $\mathcal{M}$ is either unknown or not all nodes in the sub-hypergraph satisfy the conditions for amplification induction, determining this set will become an integral part of the decision-making process for computing the amplification factor. In this case, both the subhypergraph $\mathcal{H}'$ and the set $\mathcal{M}$ must be constructed as part of the procedure. In this section, we detail the methodology that we develop for computing the subhypergraph that maximizes the amplification factor.

\subsection{Problem Formulation}

Given a hypergraph $\mathcal{H} = (\mathcal{N}, \mathcal{A})$, our goal is to construct a self-sufficient subnetwork $\mathcal{H}^{'} = (\mathcal{N}^{'}, \mathcal{A}^{'})$ and a set $\mathcal{M} \subset \mathcal{N}$ that maximizes the amplification factor. By Theorem \ref{th:1}, if the obtained subnetwork has an amplification factor greater than or equal to one, it is considered self-amplifying, and $\mathcal{M}$ would be a proper self-amplifying set. Otherwise, the subnetwork with the highest amplification factor is computed, but $\mathcal{M}$ would merely be a self-sufficient set without net-positive realizability.

We propose a mathematical optimization approach for constructing a subnetwork within a given hypergraph. The proposed mathematical program uses the following decision variables, corresponding to the main decisions to be made: 

\begin{align*}  
y_{v}= 
   \begin{cases} 
      1              & \mbox{if node $v$ is in $\mathcal{M}$}  \\
      0 & \mbox{otherwise}
   \end{cases}, \quad \forall v \in \mathcal{N}.
\end{align*}
\begin{align*}  
\mathbf{x} \in \mathbb{R}_+^{|\mathcal{A}|}: \text{ intensity vector,} \quad \alpha \in \mathbb{R}_+: \text{ amplification factor.}
\end{align*}

These variables determine which nodes belong to the distinguished subset $\mathcal{M}$ ($y$), which hyperarcs are selected in $\mathcal{H}^{'}$ ($x > 0$), the amplification factor of the subnetwork ($\alpha$), and the intensity vector ensuring the amplification factor ($\mathbf{x}$).

These variables are combined through the following inequalities to ensure the adequate construction of the desired subnetwork:
\begin{itemize}
    \item The amplification factor is well-defined:
\begin{align}\label{eq:1}
\sum_{a \in \mathcal{A}} \mathbb{T}_{va} \x_a \geq \alpha y_v \sum_{a \in \mathcal{A}} \mathbb{S}_{va} \x_a, &&\forall v \in \mathcal{N}.
\end{align}

Note that in case $v$ is chosen to be part of $\mathcal{M}$ and $a \in \mathcal{A'}$, the above constraints assure that $\mathbb{T}_\mathcal{M} \x \geq \alpha \mathbb{S}_\mathcal{M} \x$, as required to properly defined the MAF. Otherwise, either $y_v=0$, in whose case the constraint is redundant ($\sum_{a \in \mathcal{A}} \mathbb{T}_{va} \x_a \geq 0$, which is always true by the nonnegativity of $\mathbb{T}$ and the variables), or $x_a=0$, in whose case the hyperarc $a$ does not account in the computation of the MAF.

    \item The subset $\mathcal{M}$ is not empty:
    \begin{align}\label{eq:2}
        \sum_{v \in \mathcal{N}} y_v \geq 1.
    \end{align}
    To assure that the MAF is well defined, at least one node in the subhypergraph must belong to $\mathcal{M}$.
    \item The denominator in the amplification factor is non zero:
    \begin{align}\label{eq:3}
        \sum_{a\in \mathcal{A}} \mathbb{S}_{va} x_a \geq y_v, \quad \forall v \in \mathcal{N},
    \end{align}
    ensuring that if a node is selected in $\mathcal{M}$ ($y_v=1$), its contribution to the amplification factor is strictly positive. 
    \item The nodes in $\mathcal{M}$ must be self-sufficient \eqref{autonomous2}:
    \begin{align}\label{eq:4}
        y_v \leq  \Big(\sum_{a \in \mathcal{A}:\atop \mathbb{T}_{va}>0} x_a\Big) \Big(\sum_{a \in \mathcal{A}:\atop \mathbb{S}_{va}>0}  x_a\Big), \quad \forall v \in \mathcal{N}.
    \end{align}
    This condition ensures that nodes in $\mathcal{M}$ appear in at least one selected hyperarc both as a source and as a target.
    Note that this nonlinear constraint can be adequately linearized as follows:
    \begin{align}
    y_v \leq \sum_{a \in \mathcal{A}:\atop \mathbb{T}_{va}>0} x_a\\
    y_v \leq \sum_{a \in \mathcal{A}:\atop \mathbb{S}_{va}>0}  x_a,
    \end{align}
    Thus, in case $v$ is selected to be in $\mathcal{M}$, at least one hyperarc with $v$ as target and at least one are with $v$ as source are enforced to be chosen in $\mathcal{A}'$.
    \item The hyperarcs in the subnetwork must be self-sufficient \eqref{autonomous1}:
    \begin{align}\label{eq:5}
        x_{a} \leq \Delta \Big(\sum_{v \in \mathcal{N}: \atop \mathbb{S}_{va}>0} y_{v}\Big)\Big(\sum_{v \in \mathcal{N}: \atop \mathbb{T}_{va}>0} y_{v}\Big), \quad \forall a \in \mathcal{A},
    \end{align}
    where $\Delta$ is an upper bound for the components of the intensity vector $\x$. These constraints ensure that a hyperarc is only selected if there are nodes in $\mathcal{M}$ both as inputs and outputs. These nonlinear constraints can be rewritten as linear constraints as follows:
    \begin{align}
        x_a \leq \Delta \sum_{v \in \mathcal{N}: \atop \mathbb{T}_{va} >0} y_v , \quad \forall a \in \mathcal{A}, \\
        x_a \leq \Delta \sum_{v \in \mathcal{N}: \atop \mathbb{S}_{va} >0} y_v, \quad \forall a \in \mathcal{A}.
    \end{align}
    That is, a hyperarc cannot be selected ($x_a=0$) if it does not connect nodes in $\mathcal{M}$.
    
    \item The constructed subnetwork contains at least one hyperarc:
    \begin{align}\label{eq:6}
        \sum_{a \in \mathcal{A}} x_{a} \geq 1.
    \end{align}
    This condition avoids the trivial solution where no hyperarcs are selected.
\end{itemize}

The proposed mathematical optimization model is then:
\begin{align}
    \max &\; \alpha\nonumber\\
    \text{s.t. } & \eqref{eq:1}-\eqref{eq:6},\nonumber\\
    &\mathbf{x}\in \mathbb{R}^{\mathcal{A}}_+,\nonumber\\
    &y \in \{0,1\}^{\mathcal{N}},\nonumber\\
    &\alpha \geq 0.\nonumber
\end{align}

\subsection{Solution Algorithm}

As noted for problem \eqref{MGF}, this optimization model is neither convex nor concave. Furthermore, it involves binary variables, increasing its computational complexity. Despite these challenges, we propose several techniques to solve the problem exactly and efficiently.

We develop an extension of Algorithm \ref{alg:1} that allows the selection of the optimal subhypergraph and the set $\mathcal{M}$ as part of the decision. The solution strategy that we propose is an iterative approach that computes the optimal subnetwork in a finite number of steps, with guaranteed convergence to the optimal solution. The details of these algorithms are provided in Algorithm \ref{alg:3}.

In the proposed methodology, we initialize the intensity vector to $\x_0$ in iteration $\texttt{it}=0$, and compute the amplification factor for this vector for the whole network, $\alpha_\texttt{it}$. With this value, we solve the auxiliary problem:
\begin{align*}
(\rho^*, \bar \x_{it}) \in \arg\max & \,\rho\\
\mbox{s.t.} & \, \rho \leq  (\mathbb{T} \x)_v y_v - \alpha_{\it} (\mathbb{S} \x)_v y_v, &&\forall v \in \mathcal{N},\\
&\rho \geq 0,\\
& (\x, y) \in \mathcal{X}
\end{align*}
where $\mathcal{X} = \Big\{(\x, y) \in \R_+^{|\mathcal{A}^{'}|} \times \{0,1\}^{\mathcal{N}}: (\x, y) \text{ verifies } \eqref{eq:2}-\eqref{eq:6}\Big\}$. In case $\rho=0$, it implies that the obtained solution is optimal for the MAF problem, and we are done. Otherwise, the intensity vector is updated to the $\x$-solution in the above problem, and the process is repeated. 

\begin{algorithm}[h]
\caption{Computation of the subnetwork $\mathcal{H}^{'}$ and the set of self-amplifying entities $\mathcal{M}^{'}$ maximizing the amplification factor in a given hypergraph.\label{alg:3}}
\SetKwInOut{Input}{input}\SetKwInOut{Output}{output}

\Input{$\mathcal{H} = (\mathcal{N}, \mathcal{A})$, $\it =0$, and $\x_0 \in \R_+^{|\mathcal{A}^{'}|}$}
\Output{A subnetwork $\mathcal{H}^{'} = (\mathcal{N}^{'}, \mathcal{A}^{'}) \subseteq \mathcal{H}$ self-sufficient on $\mathcal{M}$ with MAF $\alpha(\mathcal{H}^{'}, \mathcal{M})$.}

Stop=False

\While{Stop=False}{
Define $\alpha^t_{\it} =  \alpha(\x_{it})$ and 
solve the linear program:\begin{align*}
(\rho, \bar \x_{it}) \in \arg\max & \,\rho\\
\mbox{s.t.} & \, \rho \leq  (\mathbb{T} \x)_v y_v  - \alpha_{\it} (\mathbb{S} \x)_v y_v, &\forall v \in \mathcal{N},\\
&\rho \geq 0,\\
& (\x, y) \in \mathcal{X}
\end{align*}
\eIf{$\rho=0$}{
Stop=True,\\
$\mathcal{A}^{'} = \{a \in \mathcal{A}: x_a \geq 0\}$,\\
$\mathcal{M} = \{v \in \mathcal{N}: y_v = 1\}$, and\\
$\alpha(\mathcal{H}^{'},\mathcal{M})=\alpha_\it$}{
$\x_{\it+1} = \bar \x_{\it}$\\
$\it=\it+1$}}
\end{algorithm}

Note that a mixed-integer linear programming problem must be solved at each iteration. Consequently, the complexity of the mathematical optimization method that we propose to extract a subnetwork from a given dynamic system with MAF is highly influenced by the number of variables and constraints involved. In particular, the time required to solve the problem depends on the number of hyperarcs and nodes in the system. To reduce the search space, we derive two strategies that allow us to accelerate the algorithm: a preprocessing procedure that removes nodes and hyperarcs that can never participate in such a subnetwork, and a divide-et-conquer approach to detect independent substructures in the given hypergraph.

\vspace*{0.5cm}
\noindent{\large \textit{Reducing the Set of Potential Nodes and Hyperarcs} \label{alg:reduction}} 
\vspace*{0.5cm}

The first dimensionality reduction that we propose to solve more efficiently Algorithm \eqref{alg:3} consists of pre-processing the nodes and hyperarcs of $\mathcal{H}$ to remove those that will never participate in the optimal subnetwork.

The procedure iteratively removes non-self-sufficient rows and columns from the source matrix $M_1$ and the target matrix $M_2$. Note that those nodes that do not appear as both sourced and targeted by the hyperarcs are not required in our search. Once these nodes are removed from the list, it may happen that some hyperarcs only act as sources or targets but not both, and then they can also be avoided in the list. This process can be iteratively applied until no hyperarcs or nodes are removed.

% This procedure is outlined in Algorithm \ref{alg:removeNullRowsAndColumns}, 

% \begin{algorithm}[H]
% %
% \caption{Algorithm that removes all non-self-sufficient rows and columns from the input and output matrix.
% \label{alg:removeNullRowsAndColumns}}
% %
% \SetKwInOut{Input}{input}\SetKwInOut{Output}{output}
% % 
% \Input{$M_{1}$, $M_{2}$
% }
% %
% \Output{$M_{1}$, $M_{2}$}
% %
% $M_{1}, M_{2} \in \mathbb{N}^{m\times n}$\\
% $rows\_to\_check = \{1, \dots, m\}$ \\
% $columns\_to\_check = \{1, \dots, n\}$ \\
% %
% \While{True}{
% $new\_null\_rows = \{\}$\\
% $new\_null\_columns = \{\}$\\
% %
% \For{i in rows\_to\_check}{
% %
% \If{
% $\sum_{i = 1}^{m}\left( M_{1} \right)_{ij} < 0$ 
% \bf{or} 
% $\sum_{i = 1}^{m}\left( M_{2} \right)_{ij} < 0$ }{
% $new\_null\_rows.insert(i)$}
% }
% %
% \For{j in columns\_to\_check}{
% %
% \If{
% $\sum_{j = 1}^{n}\left( M_{1} \right)_{ij} < 0$ 
% \bf{or} 
% $\sum_{j = 1}^{n}\left( M_{2} \right)_{ij} < 0$ }{
% $new\_null\_columns.insert(j)$}
% }
% %
% \If{$new\_null\_rows = \emptyset$ \bf{and} 
% $new\_null\_columns = \emptyset$}{
% break}
% %
% \For{$row$ in $new\_null\_rows$}{
% remove $row$ from $M_{1}$\\
% remove $row$ from $M_{2}$\\
% }
% %
% \For{$column$ in $new\_null\_columns$}{
% remove $column$ from $M_{1}$\\
% remove $column$ from $M_{2}$\\
% }
% $M_{1}, M_{2} \in \mathbb{N}^{m\times n}$\\
% $rows\_to\_check = \{1, \dots, m\}$ \\
% $columns\_to\_check = \{1, \dots, n\}$ \\
% }
% \Return $M_{1}$, $M_{2}$
% \end{algorithm}

The proposed strategy begins by initializing all rows and columns for inspection. It then iteratively detects and removes null rows and columns from the resulting incidence matrix. At each iteration, null rows are identified by checking whether the sum of entries in a given row is zero in both the input and output matrices. Similarly, null columns are detected by examining the column sums in these matrices. If no new null rows or columns are found, the algorithm terminates. When null rows or columns are detected, they are removed from both matrices, and the sets of rows and columns to be inspected are updated accordingly. This ensures that any newly created null rows or columns are considered in subsequent iterations. The process repeats until no further null rows or columns remain, at which point the simplified matrices are returned.

% The procedure begins by initializing all rows and columns for inspection (lines 1-3). It then enters a loop to iteratively detect and remove null rows and columns (lines 4-27). At each iteration, null rows are identified by checking if the sum of entries in a row is zero in both $M_1$ and $M_2$ (lines 7-11). Similarly, null columns are detected by examining the column sums in $M_1$ and $M_2$ (lines 12-16). If no new null rows or columns are found, the algorithm terminates (lines 17-19).  

% When null rows or columns are detected, they are removed from $M_1$ and $M_2$ (lines 20-27). The sets of rows and columns to inspect are updated accordingly, ensuring that any newly created null rows or columns are included in subsequent iterations. This process repeats until no further null rows or columns are found, at which point the simplified matrices are returned (line 28).

%%%
%%%
%%%

 \subsubsection*{Clustering into Connected Components}\label{subsubsec:giveMeMatrixByComponent}
\vspace*{0.5cm}
\noindent {\large \textit{Splitting the hypergraph in connected components}}
\vspace*{0.5cm}

The second strategy we employ consists of detecting whether the hypergraph contains multiple connected components. If so, the subnetwork with MAF problem can be addressed by applying Algorithm \ref{alg:3} separately to each component, and then selecting the one where the maximum MAF was obtained. Although this approach involves solving multiple instances of the problem, each instance is of smaller dimension, which leads to a significant reduction in overall CPU time.

The correct performance of this strategy is assured by the following result.
\begin{theorem}
    Let $\mathcal{G}_1=(\mathcal{N}_1,\mathcal{A}_1), \ldots, \mathcal{G}_k=(\mathcal{N}_k, \mathcal{A}_k)$ be the different connected components for a system $\mathcal{G}=(\mathcal{N},\mathcal{A})$. Let $\mathcal{G}^{'}=(\mathcal{N}^{'},\mathcal{A}^{'})$ be the MAF subnetwork for $\mathcal{G}$, with the set of self-sufficient nodes $M$. Then:
    \begin{align*}
    \alpha(\mathcal{G}) = \max_{l=1, \ldots, k} \alpha(\mathcal{G}_{l}),
    \end{align*}
    where $\mathcal{M}_l$ is the \emph{optimal} set of self-sufficient nodes for the subnetwork $\mathcal{G}_l$, for $l=1, \ldots, k$.
\end{theorem}

\begin{proof}
    The proof follows by noting that for any set $\mathcal{M} \subset \mathcal{N}^{'}$ and any intensity $\x = (\x^1, \ldots, \x^k)$, decomposed by connected components:
    \begin{align*}
    \alpha(\mathcal{G}^{'},\mathcal{M}, \x) = \min_{n \in \mathcal{M}} \frac{\dsum_{a \in \mathcal{A}^{'}} \mathbb{T}_{na} \x_a}{\dsum_{a \in \mathcal{A}^{'}} \mathbb{S}_{na} \x_a} = \min_{l=1, \ldots, k} \min_{n_l \in \mathcal{M}_l} \frac{\dsum_{a \in \mathcal{A}_{l}} \mathbb{T}_{n_la} \x^{l}_a}{\dsum_{a \in \mathcal{A}_{l}} \mathbb{S}_{n_la} \x^{l}_a} = \min_{l=1, \ldots, k} \alpha(\mathcal{G}_l, \mathcal{M}_l, \x^l).
    \end{align*}
    Thus, 
    \begin{align*}
    \alpha(\mathcal{G}^{'}) = \max_{\mathcal{M} \subset \mathcal{A}^{'} \atop \x\in \mathbb{R}^{T}_+} \alpha(\mathcal{G}^{'}, \mathcal{M}, \x) = \max_{l=1, \ldots, k} \max_{\mathcal{M}_l\subset \mathcal{A}_l \atop \x_l\in \mathbb{R}^{\mathcal{A}_l}_+} \alpha(\mathcal{G}_l,\mathcal{M}_l, \x_l) = \max_{l=1, \ldots, k} \alpha(\mathcal{G}_l).
    \end{align*}
\end{proof}

Thus, in view of the result above, we propose a divide-et-conquer approach to solve, exactly, the problem. Specifically, the proposed procedure consists of projecting the hypergraph $\mathcal{H}=(\mathcal{N},\mathcal{A})$ to a directed graph $\mathcal{G}=(V,A)$, where $V = \mathcal{N} \cup \mathcal{A}$, and 
$$
A = \{(v,a) \in \mathcal{N}\times \mathcal{A}: \mathbb{S}_{va}>0\} \cup \{(a,v) \in \mathcal{A}\times \mathcal{N}: \mathbb{T}_{va}>0\},
$$
similar to the graph represented in Figure \ref{fig:bipartite}.
That is, the hyperarcs in $H$ link nodes in $\mathcal{N}$ with the hyperarcs where they participate as source, and hyperarcs in $\mathcal{A}$ with the nodes that participate as target.

 %This strategy is outlined in Algorithm \ref{alg:giveMeMatrixByComponent}, which identifies weakly connected components in a directed bipartite graph representing nodes and hyperarcs. The graph is constructed from the source matrix $M_1$ and the target matrix $M_2$, with the algorithm outputting submatrices corresponding to each connected component.

%The graph is built by interpreting $M_1$ and $M_2$ as adjacency matrices. Nodes represent nodes and hyperarcs in the hypergraph, and directed edges are established as follows:
% \begin{itemize}
     %\item A directed edge from an node-node to a hyperarc-node is added if the corresponding entry in $M_1$ is positive.
    % \item A directed edge from a hyperarc-node to an node-node is added if the corresponding entry in $M_2$ is positive.
 %\end{itemize}

Once the graph is constructed, a traversal method is used to identify weakly connected components—subsets of nodes connected via paths, regardless of edge direction. Each component groups nodes and hyperarcs that are directly or indirectly related.

For each weakly connected component, the algorithm determines the nodes and hyperarcs it contains. It then extracts the corresponding input and output incidence submatrices, representing the interactions within the component. These submatrices, along with their associated node and hyperarc indices, form the output. The algorithm continues until all components are processed, returning a comprehensive list of submatrices for all components. This decomposition simplifies the analysis by breaking down complex networks into smaller, manageable subsystems.

\subsection{Incorporating additional features into the mathematical model}\label{sec:features}

In a self-amplifying subnetwork $\mathcal{H}' = (\mathcal{N}', \mathcal{A}')$, it is not required that every node in the set $\mathcal{N}'$ exhibits self-amplification. Instead, only the nodes in the designated \emph{self-amplifying set} $\mathcal{M} \subseteq \mathcal{N}'$ must satisfy this property. Some nodes, although part of the source and target sets of hyperarcs in the subhypergraph that is optimal under our methodology, may never become self-sufficient or satisfy the net-positive realizability condition.

Given a subhypergraph $\mathcal{H}' = (\mathcal{N}', \mathcal{A}')$, together with its optimal self-amplifying set $\mathcal{M}$ and the corresponding intensity vector $\mathbf{x}$, the nodes in $\mathcal{N}' \setminus \mathcal{M}$ can be classified into three types:

\begin{description}
    \item[\bf Source Nodes:] Nodes that appear only as sources in the hyperarcs of $\mathcal{A}'$, but never as targets. That is, all $v \in \mathcal{N}' \setminus \mathcal{M}$ such that $\mathbb{T}_{va} = 0$ for all $a \in \mathcal{A}'$ and $\mathbb{S}_{va} > 0$ for some $a \in \mathcal{A}'$.
    
    \item[\bf Sink Nodes:] Nodes that appear only as targets in the hyperarcs of $\mathcal{A}'$, but never as sources. That is, all $v \in \mathcal{N}' \setminus \mathcal{M}$ such that $\mathbb{S}_{va} = 0$ for all $a \in \mathcal{A}'$ and $\mathbb{T}_{va} > 0$ for some $a \in \mathcal{A}'$.
    
    \item[\bf Non-amplifying Nodes:] Nodes that appear both as sources and targets in $\mathcal{A}'$, but whose net balance under the intensity vector $\mathbf{x}$ is negative. That is, there exist $a_1, a_2 \in \mathcal{A}'$ such that $\mathbb{S}_{v a_1} > 0$ and $\mathbb{T}_{v a_2} > 0$, but
    $\sum_{a \in \mathcal{A}'} \mathbb{Q}_{va} \mathbf{x}_a < 0$.
\end{description}

This classification has previously appeared in the context of Chemical Reaction Networks (CRNs), where nodes represent chemical species and hyperarcs represent reactions. In that literature, the three categories above are referred to as \emph{Food}, \emph{Waste}, and \emph{Non-core} species, respectively.

Note that, with these types, source and non-amplifying nodes exhibit negative net incidence in $\mathcal{H}^{'}$, that is, for $v \in \mathcal{N}^{'}$ classified as input or non-amplifying, $\sum_{a \in \mathcal{A}^{'}} \mathbb{Q}_{va} \mathbf{x}_a < 0$. Thus, we refer to these as \emph{drain nodes} under $\mathbf{x}$, denoted by $\mathcal{N}^-$. Analogously, the sink nodes will be denoted by $\mathcal{N}^+ := \mathcal{N}' \backslash (\mathcal{M} \cup \mathcal{N}^-)$, and, together with the self-amplifying nodes, these are \emph{amplifying nodes}. Thus $\mathcal{N}' = \mathcal{M} \cup \mathcal{N}^- \cup \mathcal{N}^+$. 

With all the above, although the construction of MAF subhypergraphs is a valuable tool for identifying the hyperarcs and nodes that exhibit the highest self-amplification potential and therefore contribute most significantly to system amplification, in some applications it may be required to construct MAF subhypergraphs under specific constraints tailored to particular scenarios. In this section, we present several interesting cases and explain how they can be explicitly incorporated into the mathematical optimization model we propose for constructing MAF subnetworks. Specifically, we detail how to designate certain nodes as sources, sinks, or non-amplifying within the resulting subnetwork using linear inequalities involving the decision variables already defined in our model.

Although one can incorporate different features into the model that we already presented above, such an approach focuses on constructing $\M$ and the hyperarcs, but there are no special conditions on the source or target nodes of those hyperarcs. Thus, in order to account for those nodes not in $\M$, it is required to incorporate into our model decision variables that allow us to identify those nodes, namely:
\begin{align*}
    t_v = \begin{cases}
        1 & \mbox{if $v \in \mathcal{N}'$},\\
        0 & \mbox{otherwise.}
    \end{cases} \forall v \in \mathcal{N}
    \end{align*}
that is, $t_v=1$ if node $v$ is in the selected hypergraph (that will maximize the MAF). This set of variables is adequately defined by the following set of inequalities:
    \begin{align}
         x_a &\leq \Delta t_v, &&\forall a \in \mathcal{A}, v \in \mathcal{N} \text{ such that }  v \in T_a \cup S_a,\\
         t_v &\leq \sum_{a \in \mathcal{A}: v \in T_a \cup S_a} x_a, &&\forall v \in \mathcal{N},\\
         y_v  &\leq t_v, &&\forall v \in \mathcal{N}.
    \end{align}
    That is, if hyperarc $a$ is part of the subnetwork ($x_a>0$), then all the nodes involved in $a$ (as source or target) must be \emph{activated} to be part of the subhypergraph $\mathcal{H}'$; and if all the hyperarcs in which a node $v$ is either source or target are not in the subhypergraph, then $v$ will not be a node of such a subhypergraph. The last set of constraints assures that the set of self-amplification nodes in $\mathcal{M}$ (those with $t_v=1$ ) are chosen from the whole set of activated nodes $\mathcal{N}'$($y_v=1$). So the $\mathcal{M}$ is a subset of $\mathcal{N}'$. 
    
    Those constraints can be additionally incorporated into our model, to construct, based on its solution $(\bar y, \bar x)$, the optimal subhypergraph as $\mathcal{H}'=(\mathcal{N}',\mathcal{A}')$ with:
    \begin{align*}
    &\mathcal{N}'=\{v\in \mathcal{N}: \bar y_v =1\}, &&\mathcal{A}'=\{a\in \mathcal{A}: \bar x_a \geq 0 \}.
    \end{align*}

    With this encoding, one can require the obtained subhypergraph different conditions related to the types of nodes:

    \begin{itemize}
        \item {\bf Subhypergraphs with specific source nodes}: Particular nodes can be considered as source nodes, which are assumed to participate as source but not as targets. In case one desires to construct subhypergraphs with MAF among those that uses a fixed subset of nodes as source nodes, $\mathcal{F} \subset \mathcal{N}$. This condition can be imposed on our model with the following sets of constraints:
        \begin{align}
            &t_v = 1, &&\forall v \in \mathcal{F},\\
            &x_a = 0, &&\forall a \in \mathcal{A}: \exists v \in \mathcal{F} \text{ such that } v \in T_a
        \end{align} 
The first equations enforce that the source nodes are activated in the subhypergraph, whereas the second set of constraints avoid the use of source nodes as target of any hypearc in the subnetwork.
                \item {\bf Subhypergraphs with specific sink nodes}: Similarly to the previous consideration, some nodes are considered to be only target nodes of the hyperarcs but never source nodes. Thus, in one one desires to construct subhypergraphs with MAF among those that have a fixed sink set of nodes, $\mathcal{W} \subset \mathcal{N}$. This condition can be imposed to our model with the following sets of constraints:
        \begin{align}
            &t_v = 1, &&\forall v \in \mathcal{W},\\
            &x_a = 0, &&\forall a \in \mathcal{A}: \exists v \in \mathcal{W} \text{ such that } v \in S_a.
        \end{align}
Analogously to the previous specifications, the first equations enforce that the sink nodes are activated in the subhypergraph, whereas the second set of constraints avoid the use of those nodes as source of the hyperarcs.
\item {\bf Subhypergraphs with specific non-amplifying nodes}: In case a set of nodes  $\mathcal{B} \subset \mathcal{N}$ is assumed to be source and target of the hyperarcs in the subhypergraph but with a negative net aggregation, we can impose this condition by the following constraints:
        \begin{align}
            t_v &= 1, &&\forall v \in \mathcal{B},\\
            x_a &= 0, &&\forall a \in \mathcal{A}: \exists v \in \mathcal{B} \text{ such that } v \in T_a\backslash S_a \cup S_a\backslash T_a,\\
            \mathbb{Q}_v \x &\leq 0, &&\forall v \in \mathcal{B}.
        \end{align}
The first equations enforce that the non-amplifying nodes are activated in the subhypergraph. The second set of constraints avoid only the use of the nodes in $\mathcal{B}$ only as source or target of any hyperarc. Finally, the net aggregation under the intensity vector of these nodes is enforced to be nonpositive.
    \end{itemize}
All these constraints can be combined; that is, if $(\mathcal{F}, \mathcal{W}, \mathcal{B})$ is a triplet of desired food, waste, and non-self-amplifying desired nodes, respectively, all the above conditions can be imposed together in our model.

\section{Computational Experiments} \label{sec:experiments}

To validate the performance of the mathematical optimization models, we conducted a series of computational experiments. These experiments also enable the comparison of the performance of the proposed algorithms on the same input data. 

We run a series of computational experiments designed to evaluate the performance of the proposed algorithms across various synthetic instances. These instances are carefully generated using the Python software \texttt{SBbadger} proposed by \cite{sbbadger} to reflect diverse problem characteristics, ensuring a thorough assessment of each algorithm’s efficiency, scalability, and robustness of random hypergraphs simulating chemical reaction networks. The primary objective of these experiments is to validate the proposed method for computing the MAF of a given subhypergraph. The results provide insights into the strengths and limitations of each algorithm and guide further optimization efforts.

We generated hypergraph-based systems of different sizes as follows: for each $n \in \{10, 20, 30, 40, 50, 100\}$, we randomly simulated $8$ instances with $|\mathcal{A}|=n$, using the function \texttt{generate\_serial.models}. The software provides the incidence matrices $\mathbb{T}$ and $\mathbb{S}$ required to compute the MAF using the different methodologies. The generated files are available in the GitHub repository \url{github.com/anticiclon/self-amplifying-hypergraphs}.%\url{github.com/vblancoOR/self-amplifying-hypergraphs}.

We ran our algorithms until they identified the optimal solutions of the MAF problems. However, in the context of searching for self-amplifying subhypergraphs, Theorem~\ref{th:1} ensures that—even if the solution is not yet proven optimal—any iteration of Algorithm~\ref{alg:1} or Algorithm~\ref{alg:3} that yields a value of $\alpha > 1$ produces a subnetwork that satisfies the self-amplification conditions. Therefore, in our experiments, in addition to recording the CPU time required to certify optimality, we also report the first self-amplifying subhypergraph identified during the process, along with the time taken to find it.

We applied our approaches to all the generated instances, using Gurobi 11.03 as the optimization solver on an Ubuntu 22.04.03 environment with an AMD EPYC 7042p 24-Core Processor and 64 GB RAM. The maximum number of iterations for the algorithm was set to 1000.

%Apart from Algorithms \ref{alg:1} and \ref{alg:3} we ran an intermediate approach where all the hyperarcs in $\mathcal{H}$ are required in the subhypergraph, but the self-amplifying nodes in $\M$ are found by the mathematical optimization model. For this approach, we modify the mathematical optimization model of Algorithm \ref{alg:3} by fixing the $z$ variables to $1$, and then the unique binary decisions are given by the $y$-variables (those determining the self-amplifying nodes in $\M$). We denote this approach as Algorithm $1'$.

For each instance, we recorded the following metrics for Algorithms \ref{alg:1} and \ref{alg:3}: the amplification factor obtained at each iteration until convergence, the CPU time required for each iteration, and the corresponding feasible intensity vector ($\x$) achieving the amplification factor. Additionally, for Algorithm \ref{alg:3}, we also collect the nodes and hyperarcs comprising the optimal subhypergraph.

The detailed results obtained for the synthetic instances are available in the GitHub repository 
\url{github.com/anticiclon/self-amplifying-hypergraphs}. %\url{github.com/vblancoOR/self-amplifying-hypergraphs}

Table \ref{table_results} presents the averaged results of the experiments. For each number of nodes in the system (first column) and each algorithm (second column), the table shows the average number of iterations and CPU time (in seconds), listed in columns \texttt{Opt\_it} and \texttt{Opt\_Time}, respectively—required to reach the optimal MAF. We also show the average number of iterations and CPU time (in seconds) required to identify the first self-amplifying subnetwork, shown in columns \texttt{First\_it} and \texttt{First\_Time}.

As observed, the computational demands increase as the number of decisions required by the algorithms grows. As expected, identifying the optimal subhypergraph is more challenging and time-consuming than other approaches. Nevertheless, while obtaining the optimal solution may be computationally expensive, detecting a self-amplifying subhypergraph (and determining whether one exists) can be achieved within reasonable CPU time.

\begin{table}[h!]
\centering
\begin{tabular}{| l | l | l | r | l | r |}
\hline
\# Nodes & Alg. & \texttt{First\_it} &  \texttt{First\_Time} (secs.) & \texttt{Opt\_it} & \texttt{Opt\_Time} (secs.) \\
\hline
$10$ & \ref{alg:1} & 3 & 0.02 & 13 & 0.09 \\
 & \ref{alg:3} & 2 & 0.04 & 6 & 0.09 \\
\hline
$20$ & \ref{alg:1} & 5 & 0.14 & 19 & 0.45 \\
& \ref{alg:3} & 3 & 0.09 & 10 & 1.97 \\
\hline
$30$ & \ref{alg:1} & 5 & 0.29 & 20 & 0.97 \\
 & \ref{alg:3} & 3 & 0.40 & 11 & 38.37 \\
\hline
$40$ & \ref{alg:1} & 4 & 0.31 & 14 & 1.13 \\
 & \ref{alg:3} & 3 & 0.61 & 9 & 87.33 \\
\hline
$50$ & \ref{alg:1} & 5 & 2.67 & 24 & 13.53 \\
 & \ref{alg:3} & 4 & 26.50 & 11 & 2301.04 \\
\hline
$100$ & \ref{alg:1} & 7 & 1.75 & 13 & 3.56 \\
 & \ref{alg:3} & 3 & 4.23 & 10 & 1723.09 \\
\hline
\end{tabular}
\caption{Average results of the synthetic experiments.\label{table_results}}
\end{table}

As previously mentioned, detecting and enumerating all self-amplifying subhypergraphs is a challenging  problem (NP-complete). However, the use of the MAF appears to mitigate the computational cost of these tasks, as illustrated in Figure \ref{fig:times}. There, we show the boxplots, for each instance size, of the CPU times (in log scale to ease the comparison) required to both detect a first self-amplifying subhypergraph (MAF greater than one) and compute the optimal MAF, with the three approaches. For Algorithm \ref{alg:1} (left plot), the green boxes indicate the CPU times required to detect, in a first iteration, that the given hypergraph is self-amplifying (in case it is), whereas the orange boxes indicate the times to compute the MAF of the hypergraph. 
%For Algorithm 1' (center plot), the green boxplots represent the time to find, for the first time, the self-amplifying nodes that reach a amplifying factor greater than one, and the orange ones, to compute its optimal MAF. 
The left plot shows the results for Algorithm \ref{alg:3}. There, the green boxplots show the times required to detect that a self-amplifying subhypergraph exists for the hypergraph, whereas the orange boxes show the times required to find the optimal subhypergraph.

As can be observed, in all cases, certifying the optimality of the MAF may require, in the worst-case scenario, hours of CPU time. In contrast, if we focus on the time required in the iterative procedure to find, for the first time, a subhypergraph with an MAF greater than $1$, the procedure took only a few seconds to identify a self-amplifying subnetwork. Thus, if the goal is to detect that a hypergraph is self-amplifying or to find a subhypergraph of the given hypergraph that is self-amplifying, the difficulty of the problem reduces significantly. 

\begin{figure*}[h]
\includegraphics[width=0.49\linewidth]{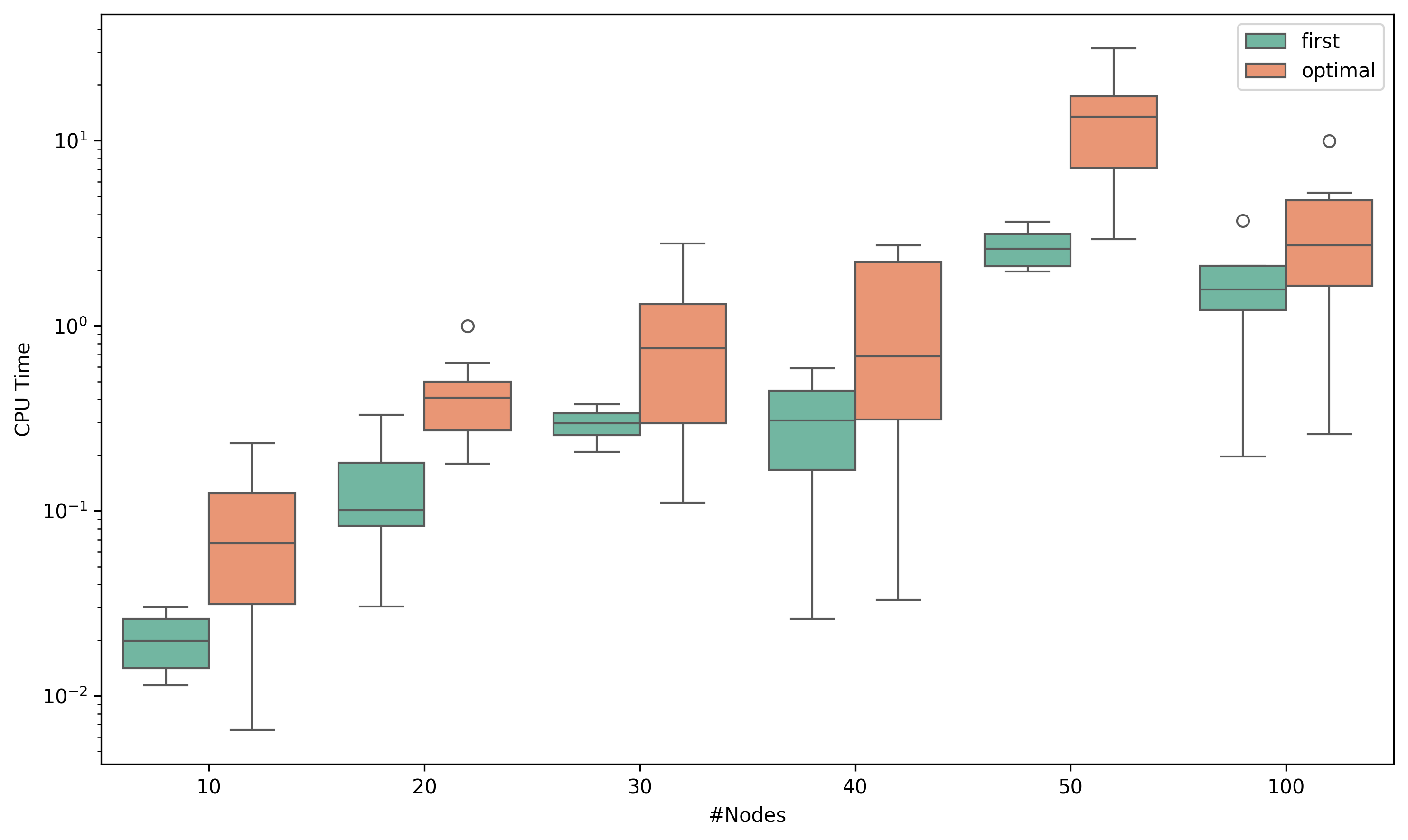}~\includegraphics[width=0.49\linewidth]{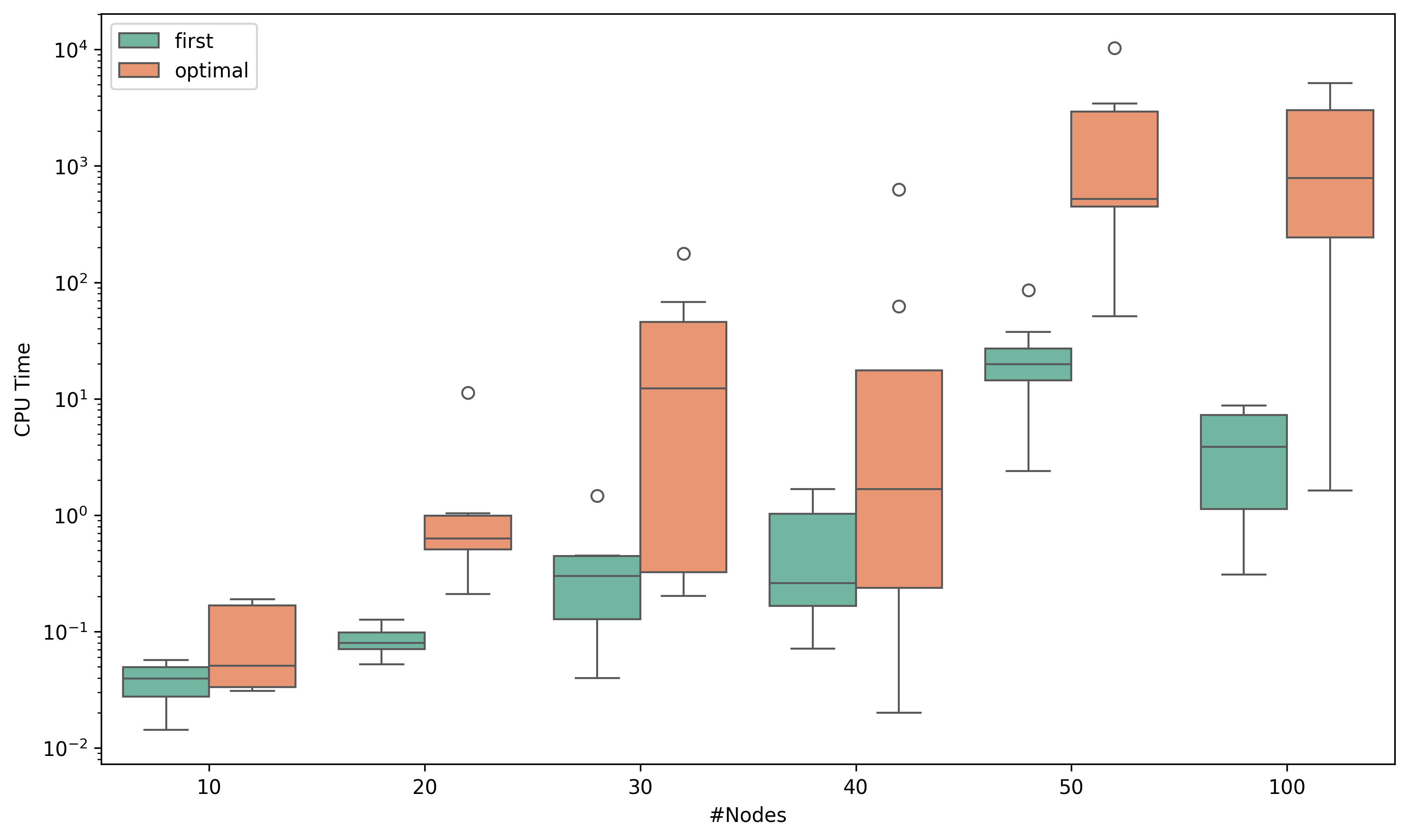}
    \caption{Boxplots of CPU times (in log scale) for the difference sizes of the instances. The green boxplots indicate the CPU times to reach the first self-amplifying network ($\alpha\geq 1$). The orange boxplots indicate the times to find the subhypergraph maximizing the value of $\alpha$.\label{fig:times}}
\end{figure*}

The second interesting observation drawn from our experiments is that although the MAF increases with the number of iterations of our algorithm, it stabilizes to a close-to-optimal value in a few iterations. This means that one could find a high-quality subhypergraph (in terms of the MAF) in a relatively small number of iterations and CPU time. In Figures \ref{fig:gap_alpha1} and \ref{fig:gap_alpha2}, we show the gap of the solution obtained with respect to the optimal MAF, in terms of the CPU time (left) and the number of iterations (right), for the instances of size $50$ (a similar behavior is observed for the other sizes), when Algorithms \ref{alg:1} and \ref{alg:3} are applied, respectively.

\begin{figure*}[h!]
\includegraphics[width=0.5\textwidth]{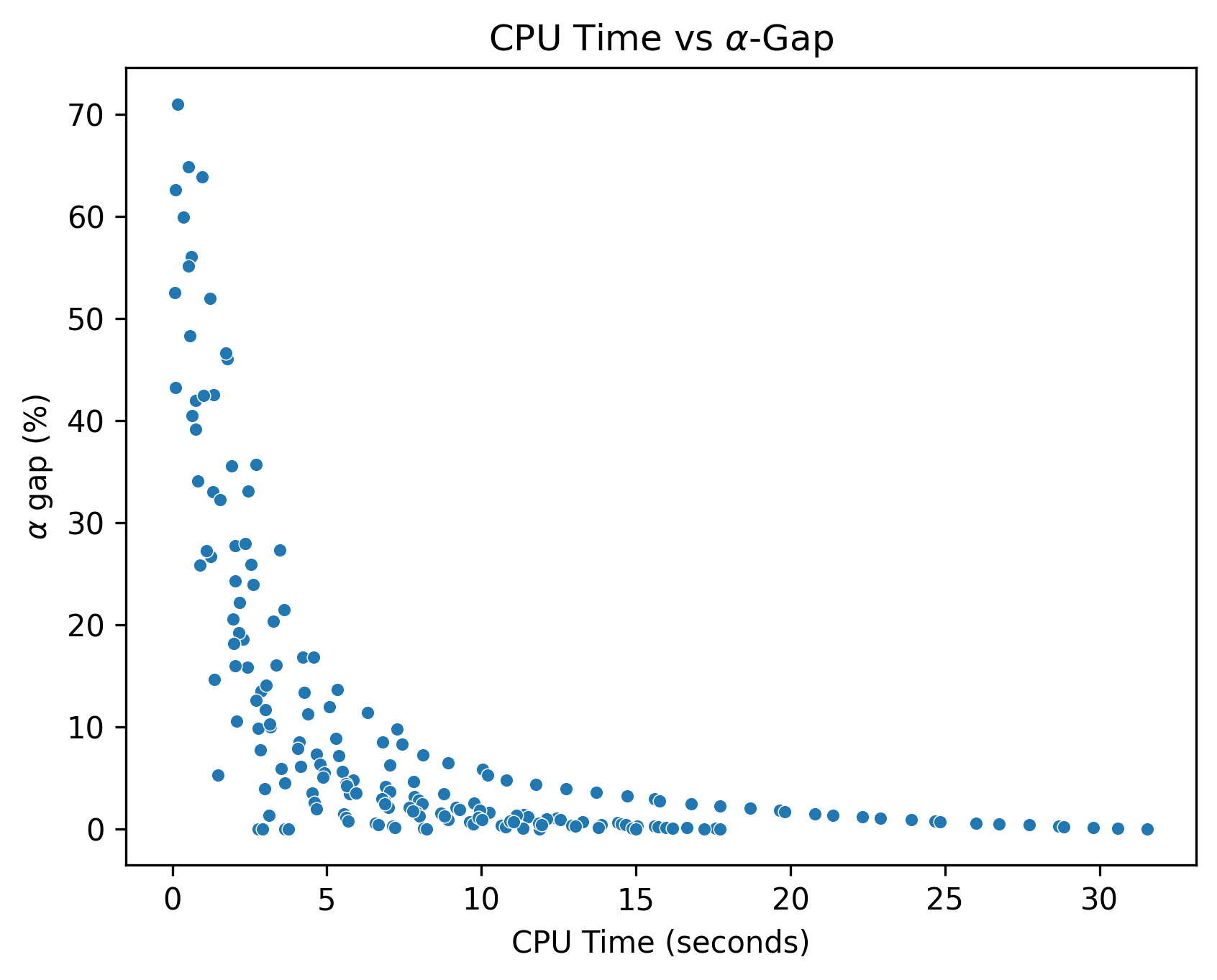}~\includegraphics[width=0.5\textwidth]{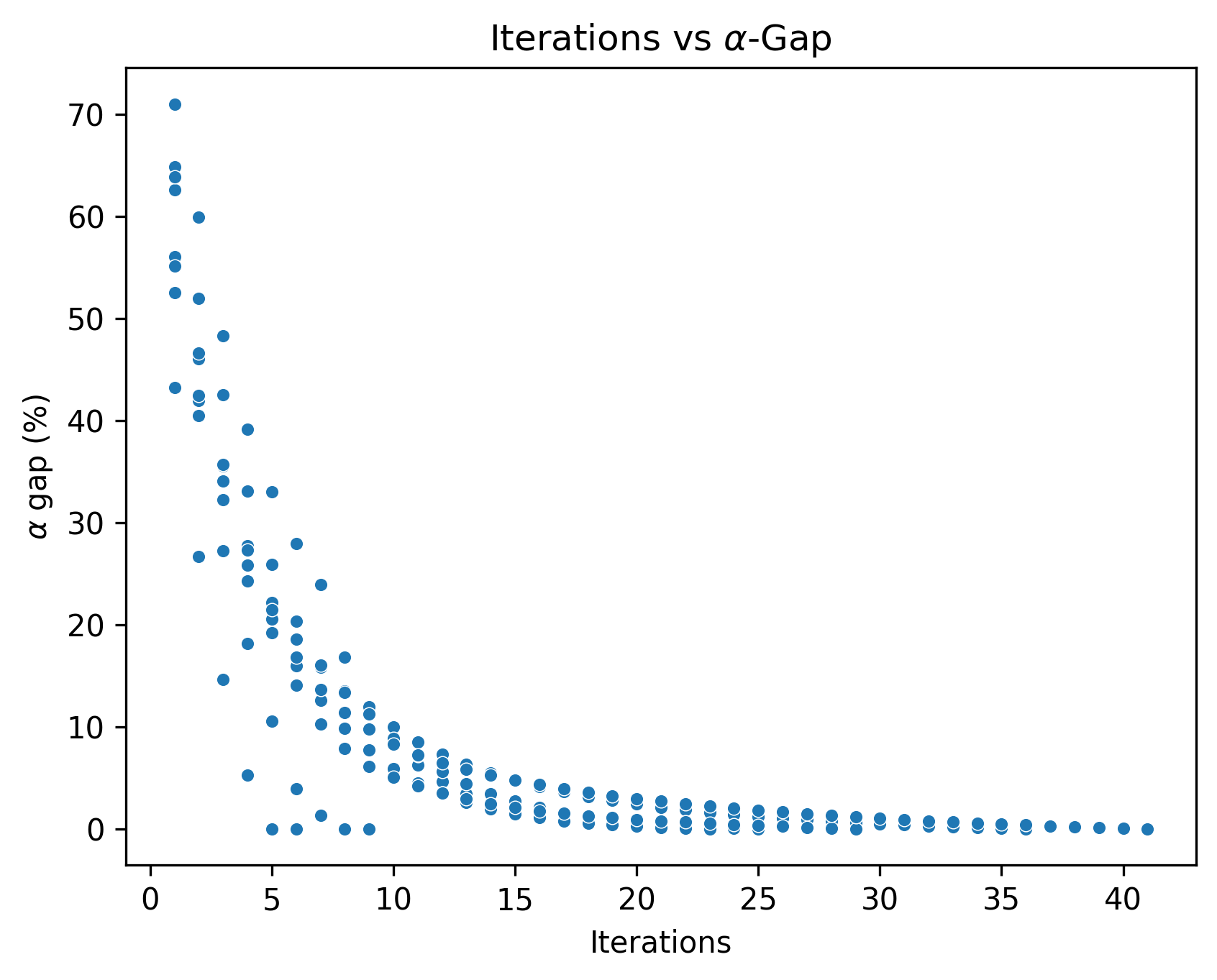}
    \caption{Gap with respect to the optimal value of the MAF with Algorithm \ref{alg:1} in terms of CPU time (left) and the number of iterations (right).
    \label{fig:gap_alpha1}}
\end{figure*}

\begin{figure*}[h!]
\includegraphics[width=0.5\textwidth]{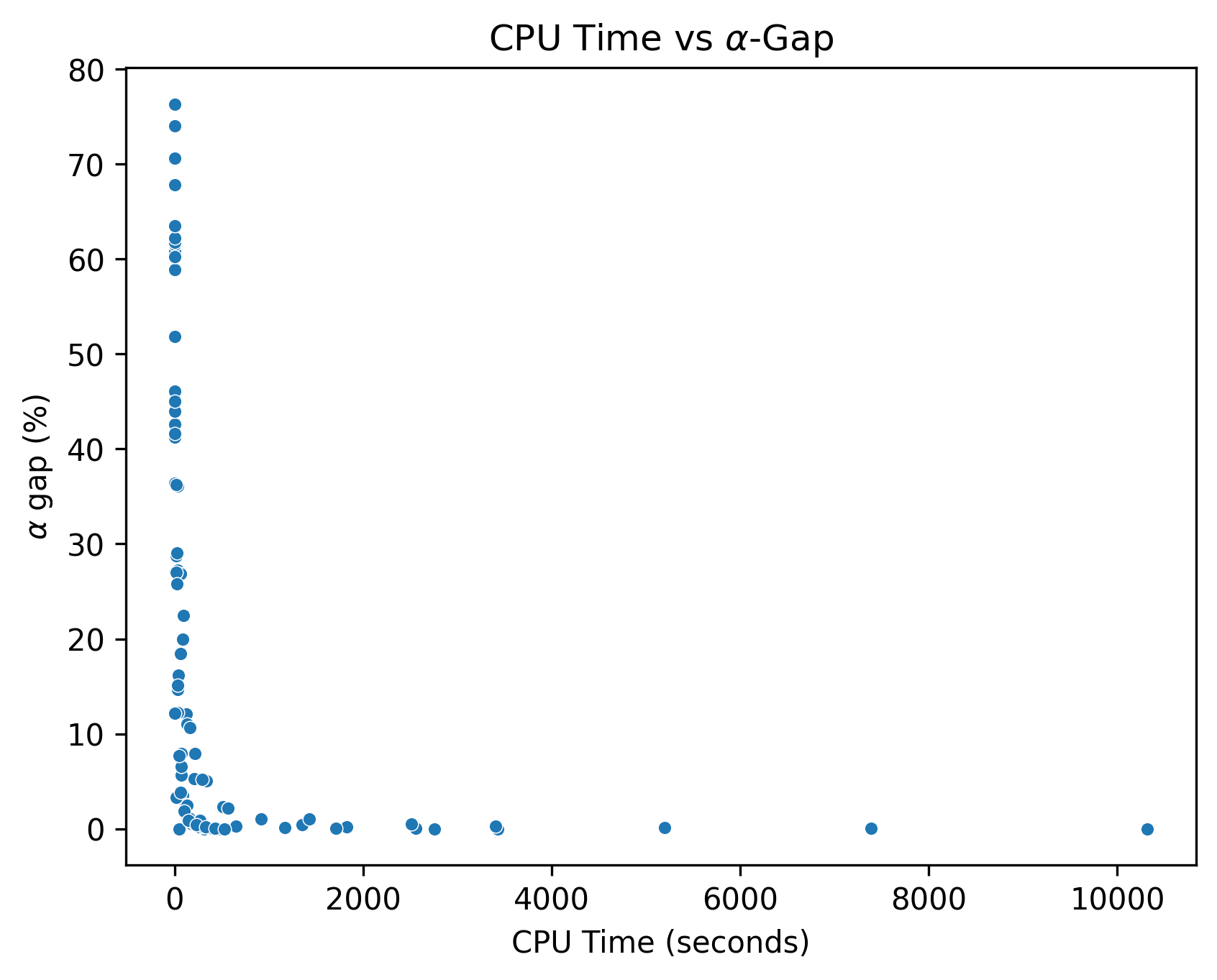}~\includegraphics[width=0.5\textwidth]{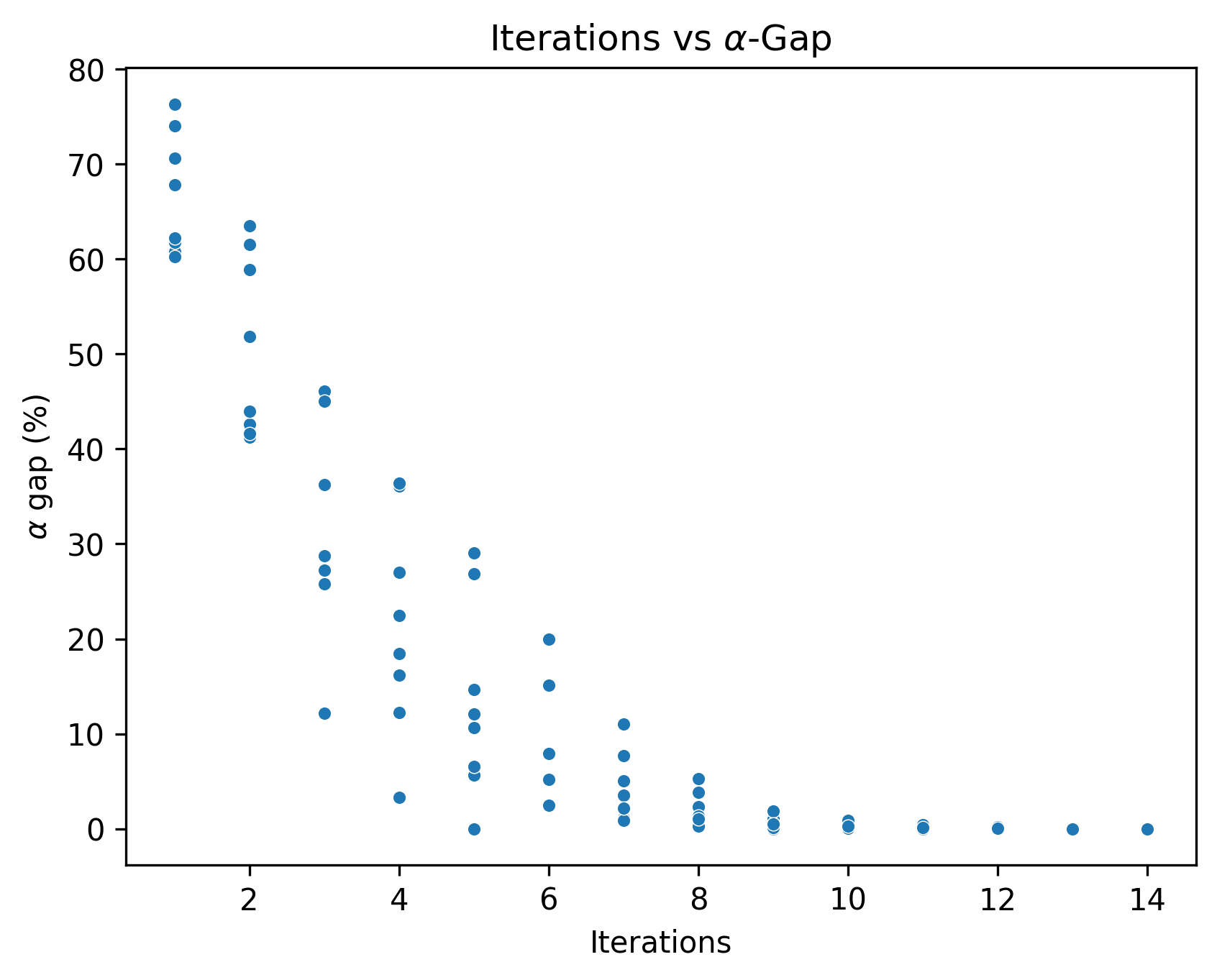}
    \caption{Gap with respect to the optimal value of the MAF with Algorithm \ref{alg:3} in terms of CPU time (left) and the number of iterations (right).
    \label{fig:gap_alpha2}}
\end{figure*}

In both algorithms, in the first iterations, the initial gap is considerably reduced from $70-80\%$ to less than $20\%$. This performance can also be observed in Figure \ref{fig:thresholds} where we represent the percent of instances that reached gaps smaller than thresholds $5\%$, $10\%$, $15\%$, $20\%$, and $50\%$ within a certain CPU time for the three approaches (Algorithm \ref{alg:1}: left, and Algorithm \ref{alg:3}: right). The performance of Algorithm \ref{alg:3}, as expected, is different, and it requires more time to stabilize, although in $10$ minutes half of the instances reduce the gap from $70-80\%$ to $20\%$.

\begin{figure}[h]
\includegraphics[width=0.49\textwidth]{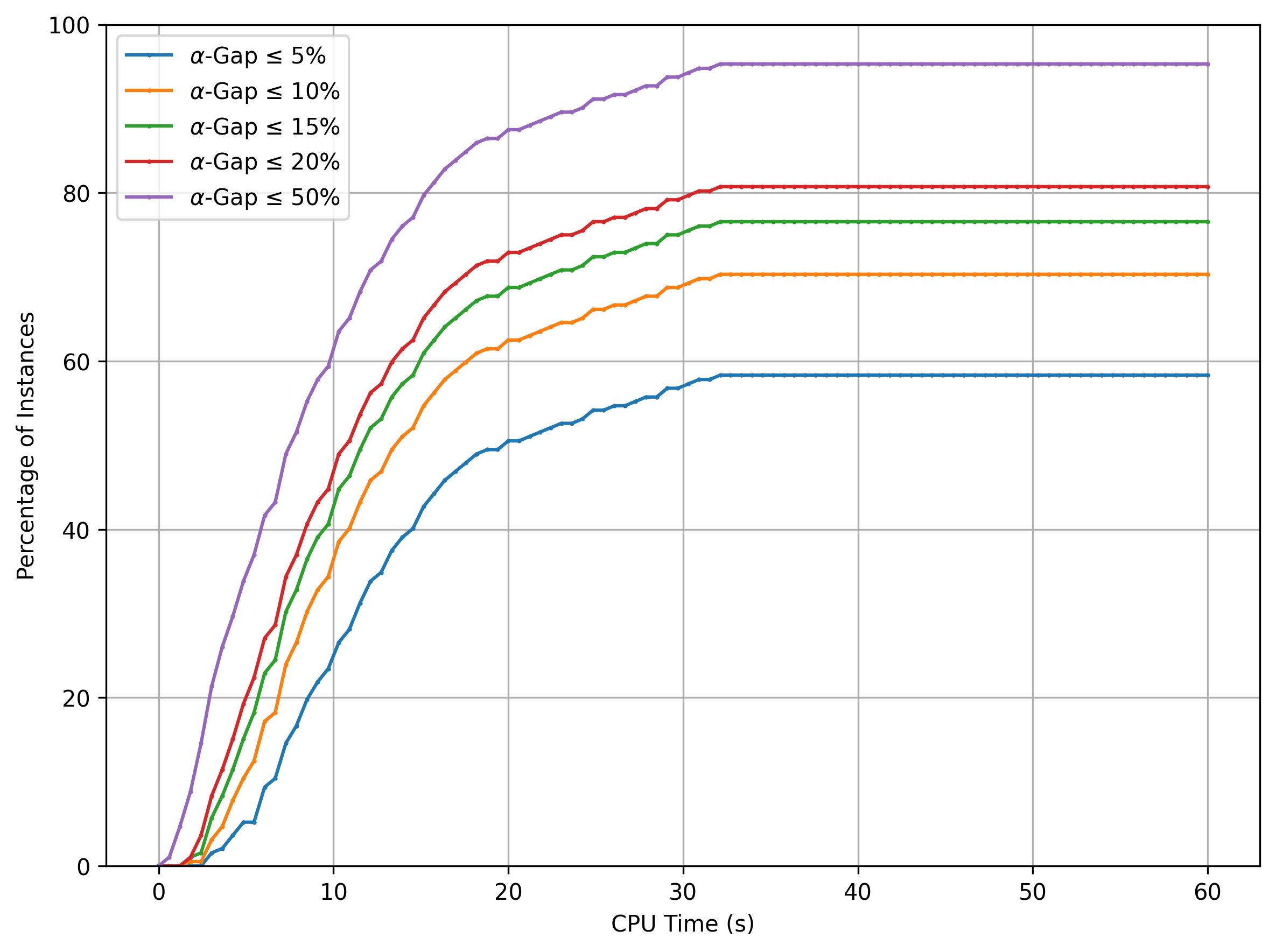}~\includegraphics[width=0.49\textwidth]{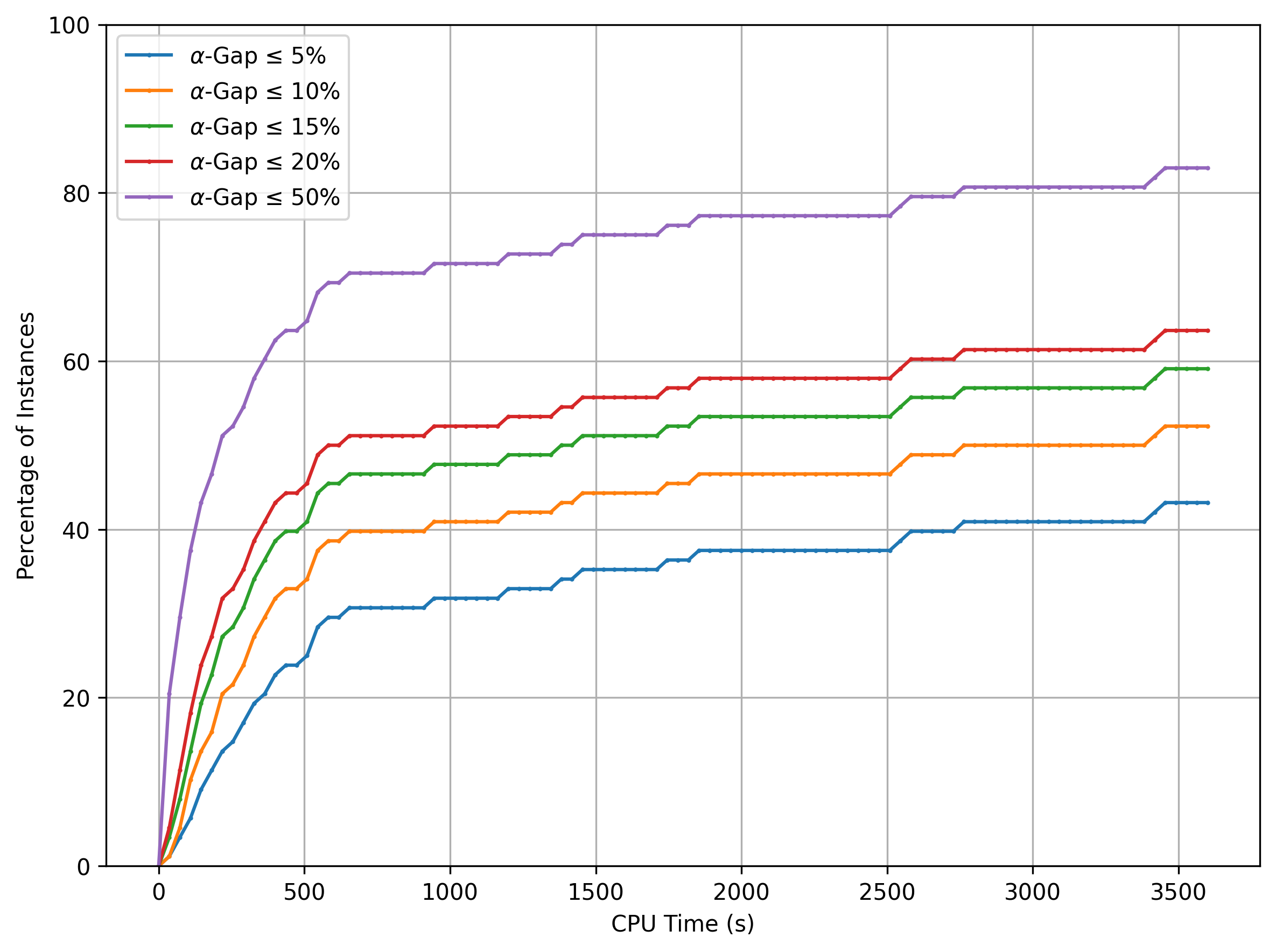}
\caption{Performance of the $\alpha$ gap by time for the different approaches. \label{fig:thresholds}}
\end{figure}

% In Figure \ref{fig:yyy}, we present the average CPU time required for each tested number of nodes. This information is further summarized in Figure \ref{fig:zzz}, where the performance of all algorithms is plotted in the same figure, using a log scale for CPU time to facilitate interpretation. While Algorithms \ref{alg:1} and $1'$ require comparable computation times, Algorithm \ref{alg:3} is significantly more demanding in terms of time.

% Amplification factor values at each iteration for Algorithm \ref{alg:1} (left) and Algorithm \ref{alg:3} (right)

% \begin{figure*}[h!]
%     \includegraphics[width=0.33\linewidth]{figures/time_alg_1.png}~    \includegraphics[width=0.33\linewidth]{figures/time_alg_half.png}~
%     \includegraphics[width=0.33\linewidth]{figures/time_alg_2.png}
%     \caption{Variation in iteration time with the number of nodes for Algorithm \ref{alg:1} (left), Algorithm $1'$ (center), and Algorithm \ref{alg:3} (right). \label{fig:yyy}}
% \end{figure*}

% \begin{figure*}[h!]
%     \centering\includegraphics[width=0.7\linewidth]{figures/time_log.png}
%     \caption{Total CPU time required by Algorithms \ref{alg:1}, $1'$, and \ref{alg:3} to solve the instances, plotted for comparison. \label{fig:zzz}}
% \end{figure*}

\section{Application: Autocatalytic Subnetworks in Chemical Reaction Networks}
\label{sec:case_study}

In this section, we apply the proposed framework for identifying and analyzing self-amplifying subhypergraphs in two real-world instances where hypergraph naturally arise, in Chemical Reaction Networks (CRNs). We study two well-known datasets: the \emph{Formose reaction network} and the \emph{E. coli core metabolism network}.

CRNs can be modeled as multi-directed hypergraphs, where nodes represent chemical species and hyperarcs represent reactions that consume and produce these species. Autocatalytic subhypergraphs, in which certain species catalyze their own production, are essential for understanding self-replication and amplification dynamics in chemical and biological systems. We apply our optimization-based methods, specifically Algorithms~\ref{alg:1} and~\ref{alg:3}, to detect such substructures and compute their Maximal Amplification Factors (MAFs).

Formally, a CRN is represented as a hypergraph $\mathcal{H} = (\mathcal{N}, \mathcal{A})$, where $\mathcal{N}$ denotes the set of species and $\mathcal{A}$ the set of reactions. Each reaction $a \in \mathcal{A}$ is modeled as a hyperarc $a = (S_a, T_a)$, with $S_a \subseteq \mathcal{N}$ as the source set (reactants) and $T_a \subseteq \mathcal{N}$ as the target set (products). The input and output incidence matrices, $\mathbb{S}$ and $\mathbb{T}$, encode the multiplicities of species consumed and produced, respectively. Our objective is to identify subhypergraphs $\mathcal{H}' = (\mathcal{N}', \mathcal{A}') \subseteq \mathcal{H}$ that exhibit self-amplification. Our notion of self-amplifying subhypergraphs coincides with some of the notions of autocatalytic subnetwork that have been proposed in the literature~\cite[see e.g.][]{gagrani2023geometry}.

The two CRNs that we analyze here have been shown to play a crucial role in understanding biological and chemical evolution. Both networks exhibit complex interrelationships and include well-known autocatalytic processes, making them particularly suitable for evaluating the practical effectiveness of our methods.

The first network is the \textbf{Formose Reaction Network}, which models the prebiotic synthesis of simple sugars, such as ribose, from formaldehyde. This network comprises 29 species and 38 reactions, and it features well-characterized autocatalytic cycles that are widely regarded as potential precursors to life. Due to its relatively compact size and well-understood behavior, the Formose network is an excellent candidate for detailed analysis.

The second network is the \textbf{E. coli Core Metabolism Network}, derived from the BiGG Models platform \cite{king2016bigg}. This network captures the core metabolic processes of Escherichia coli, a widely studied model organism in microbiology. It consists of 72 species and 95 reactions, encompassing the fundamental pathways required for cellular metabolism. Its larger scale and biological complexity make it a strong benchmark for assessing the scalability and interpretability of our analytical framework.

We applied three algorithms to both networks as part of our analysis. First, Algorithm~\ref{alg:1} was used to compute the maximum autocatalytic flow (MAF) for a given self-sufficient subhypergraph $\mathcal{H}$, as detailed in Section~\ref{sec:computing_MAF}. Second, we performed an exhaustive enumeration of autocatalytic cores using the method described by \citet{gagrani2023geometry}. For each identified core, we then applied Algorithm~\ref{alg:1} to evaluate its MAF. Finally, Algorithm~\ref{alg:3} was used to identify the subhypergraph $\mathcal{H}' \subseteq \mathcal{H}$ that achieves the highest MAF. This approach jointly optimizes the selection of both species and reactions, as explained in Section~\ref{sec:optimal_subhypergraphs}.

In what follows, we summarize our findings for both networks below. Chemical networks are visualized as bipartite graphs, with circles representing species and squares representing reactions. For simple reactions involving one-to-one conversions, we draw direct arcs between species.

\subsection*{Formose Reaction Network}

Following the application of the reduction procedure (Algorithm \ref{alg:reduction}), Algorithm \ref{alg:1} was applied to the resulting self-sufficient subhypergraph, which consisted of 27 species. The algorithm identified a maximum autocatalytic factor (MAF) of 1.130 in 2.09 seconds across 21 iterations. The resulting intensity vector $\mathbf{x}$ concentrated the flow among 7 species and 7 reactions, aligning with a known autocatalytic core, as depicted in the left panel of Figure~\ref{fig:formose_results_1}.

To further explore the network's autocatalytic capacity, we enumerated 38 distinct autocatalytic cores and computed their MAFs using Algorithm~\ref{alg:1}. The top three cores, based on MAF, are shown in the left panel of Figure~\ref{fig:formose_results_2}, while the distribution of MAFs across all cores appears in the right panel. Additionally, Figure~\ref{fig:formose_results_3} demonstrates the inverse relationship between core strength and size, with the strongest cores typically associated with smaller subhypergraphs in terms of both species and reactions.

Using Algorithm~\ref{alg:3}, we identified a particularly efficient substructure consisting of 7 species and 7 reactions, yielding a MAF of 1.138. This result was obtained in just 1.6 seconds over 5 iterations and corresponds to the strongest known autocatalytic core, as shown in the right panel of Figure~\ref{fig:formose_results_1}.

\begin{figure}[h!]
    \centering
    \includegraphics[width=0.495\textwidth]{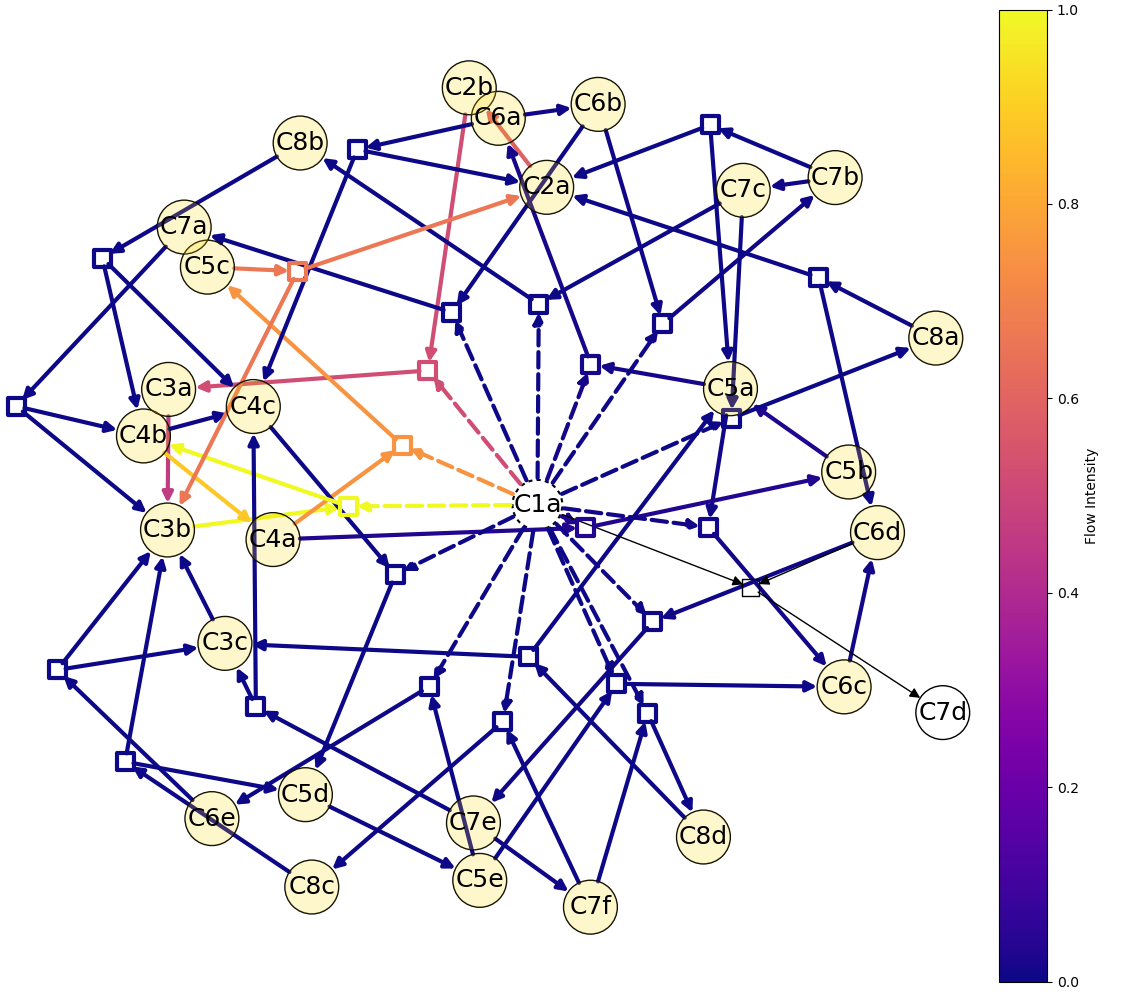}
    \includegraphics[width=0.495\textwidth]{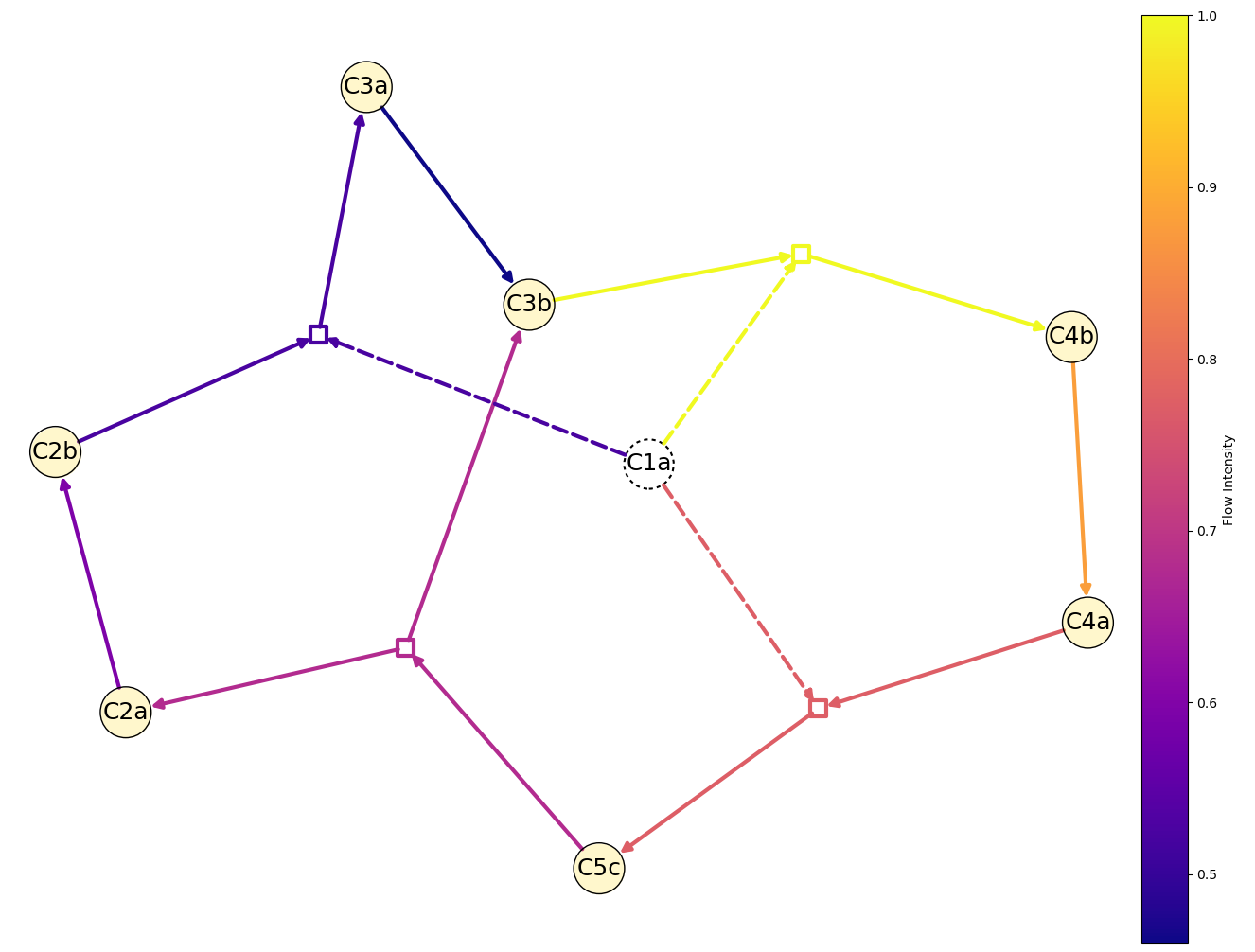}
    \caption{Application of Algorithms~\ref{alg:1} and~\ref{alg:3} to the Formose network. The left panel shows the largest self-sufficient subnetwork identified by Algorithm~\ref{alg:1}. The right panel depicts the autocatalytic subnetwork with the highest MAF (\textit{strongest subnetwork}), as determined by Algorithm~\ref{alg:3}. Notably, this strongest subnetwork corresponds to the autocatalytic core with the highest MAF (also highlighted in green in the left panel of Figure~\ref{fig:formose_results_2}). The legend indicates the optimal flow-intensity assigned to each reaction within the subnetwork.}
    \label{fig:formose_results_1}
\end{figure}

\begin{figure}[h!]
    \centering
    \includegraphics[width=0.4\textwidth]{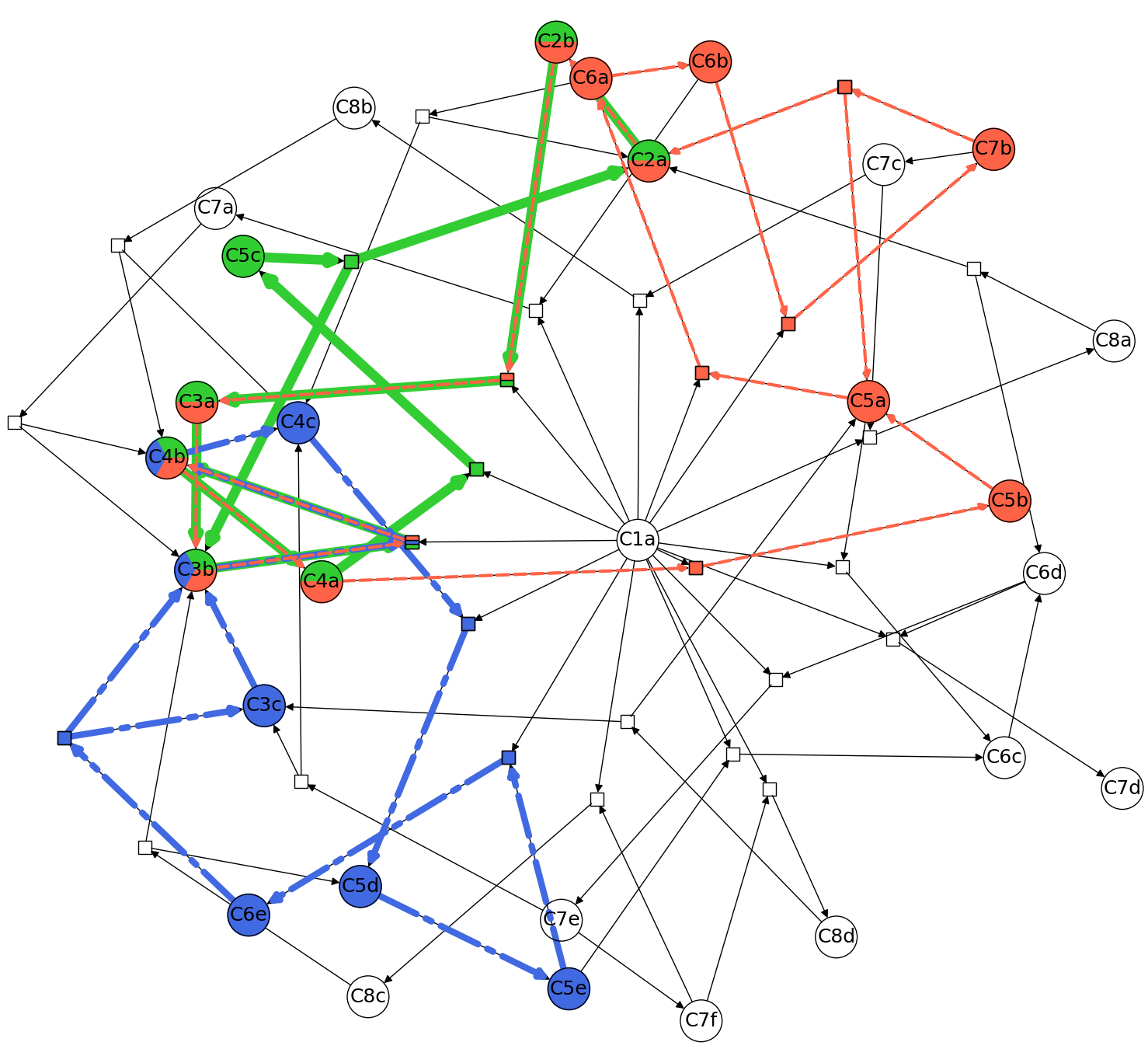}
    \includegraphics[width=0.59\textwidth]{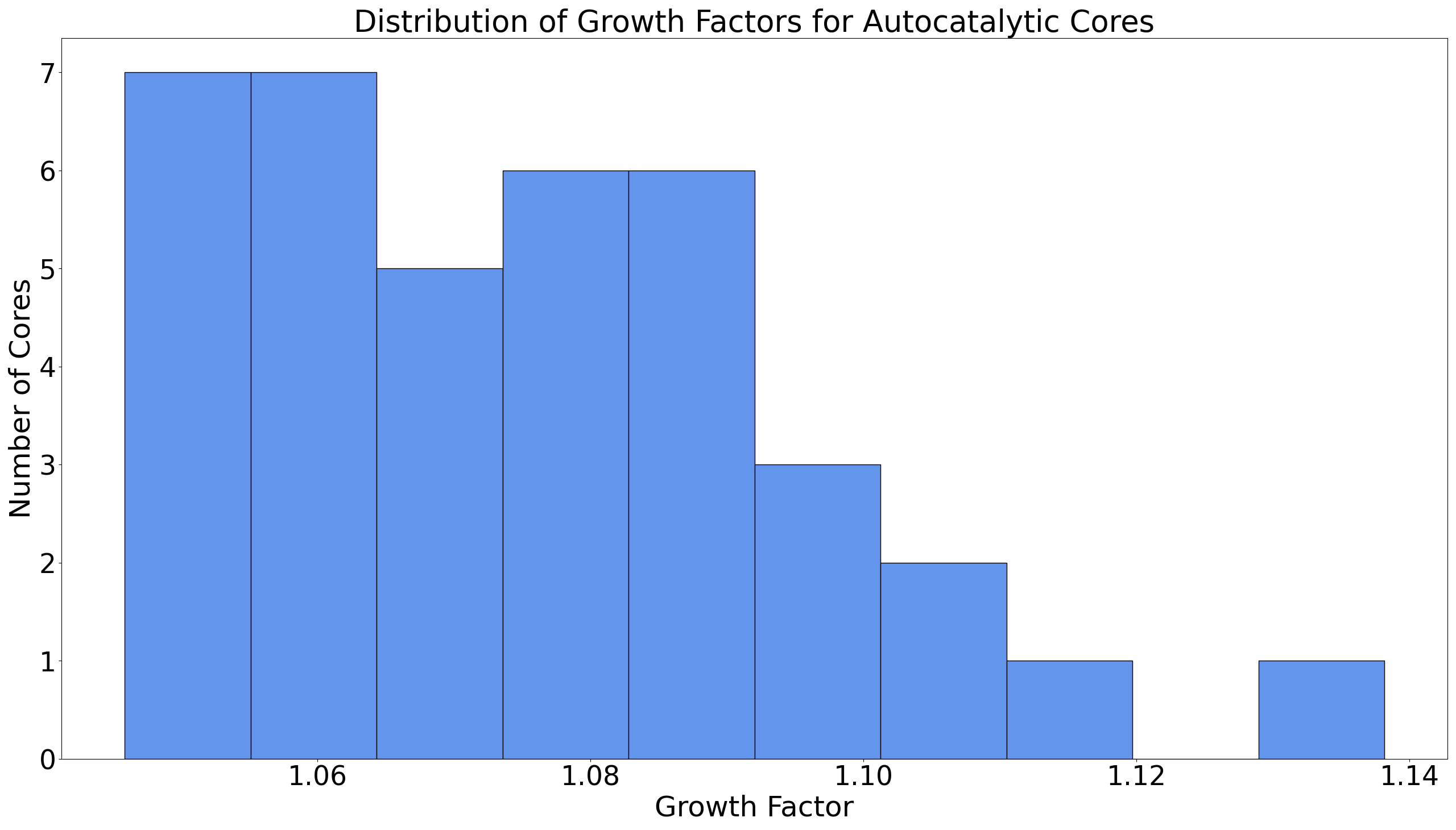}
    \caption{Autocatalytic cores in the Formose network. The left panel shows the top three cores ranked by MAF, colored green ($1.138$), blue ($1.112$), and red ($1.105$). Details of these cores are provided in Appendix~\ref{app:tables}. The right panel presents a histogram of the MAFs for all 38 autocatalytic cores identified in the network.}
    \label{fig:formose_results_2}
\end{figure}

\begin{figure}[h!]
    \centering
    \includegraphics[width=1\textwidth]{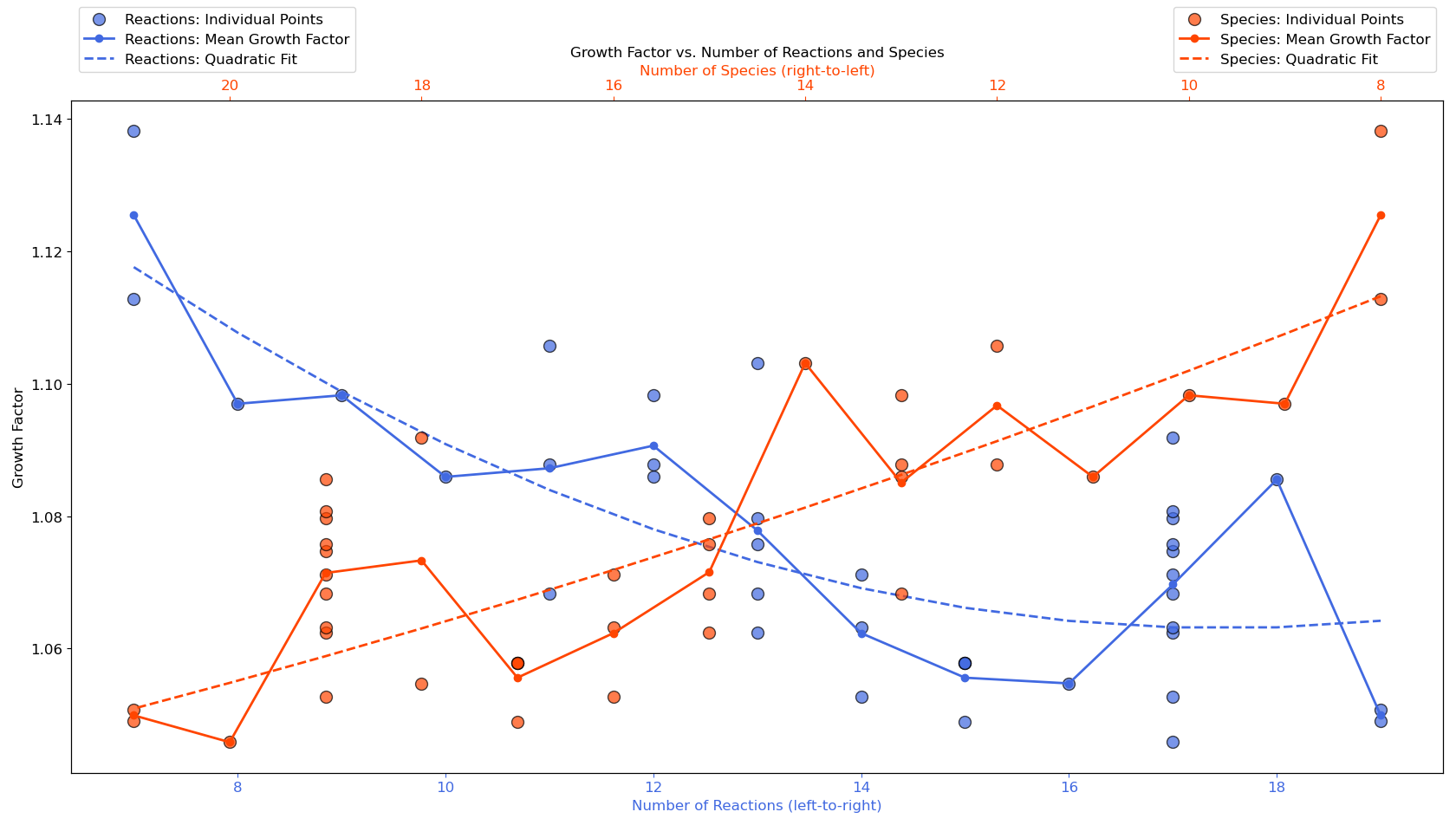}
    \caption{Relationship between the MAF and the structural complexity of autocatalytic cores in the Formose network. The plot shows how the number of species and reactions in a core influences its amplification potential.}
    \label{fig:formose_results_3}
\end{figure}

\subsection*{E. coli Core Metabolism Network}

 After reduction, the E. coli core metabolism network yielded a self-sufficient subhypergraph containing 25 species. When Algorithm~\ref{alg:1} was applied, it produced a MAF of 1.0 in 1.55 seconds over 6 iterations. This indicates a self-sustaining yet non-amplifying structure, characterized by the condition $\mathbb{Q}\mathbf{x} = 0$ for the subnetwork $\mathcal{M}$. The identified subnetwork is shown in the left panel of Figure~\ref{fig:ecoli_results_1}.

A more exhaustive analysis involved the enumeration of 581 autocatalytic cores within the network. Their MAF values are visualized in the histogram in Figure~\ref{fig:ecoli_results_2} (right), while the left panel highlights the top three cores, each achieving a MAF of 1.89. As already observed in the Formose network, the data in Figure~\ref{fig:ecoli_results_3} reinforces the trend that smaller cores tend to attain higher MAF values.

Applying Algorithm~\ref{alg:3} to this network revealed a high-performing subhypergraph comprising 13 species and 13 reactions, achieving a MAF of 2.77. This solution was obtained over 22 iterations in 234.58 seconds. Unlike the Formose case, where the result matched a known core, this structure emerges from synergistic interactions across multiple cores and supplemental reactions, as illustrated in the right panel of Figure~\ref{fig:ecoli_results_1}.

In Figure \ref{fig:ecoli_zoom_subnet} we plot the optimal subhypergraph obtained with Algorithm \ref{alg:3}. In the left plot, we draw the MAF subhypergraph. The self-amplifying nodes ($\M$) are shown in solid circles, and every other node in the subhypergraph is shown in dotted circles. Applying the algorithm proposed by \cite{gagrani2023geometry}, we listed all the \emph{autocatalytic cores} cores contained in such a subhypergraph, and hightlighted them in different colors. The main conclusion for this study is that \emph{combining} cores one may obtain better amplifying subhypergraphs.  In the right panel, the nodes (species) in the optimal subhypergraph  are decomposed into source, sink, non-amplifying, and amplifying sets (see Section~\ref{sec:features}). We also indicate which of the above cores, if any, contains each self-amplifying node (Green-C1, Blue-C2, Orange-C3, Purple-C4, Yellow-C5). The edges in the strongest subnetwork along with the reactions of self-amplifying subhypergraph contained inside of them can be found in Appendix \ref{app:tables}.

\begin{figure}[h!]
    \centering
    \includegraphics[width=0.53\textwidth]{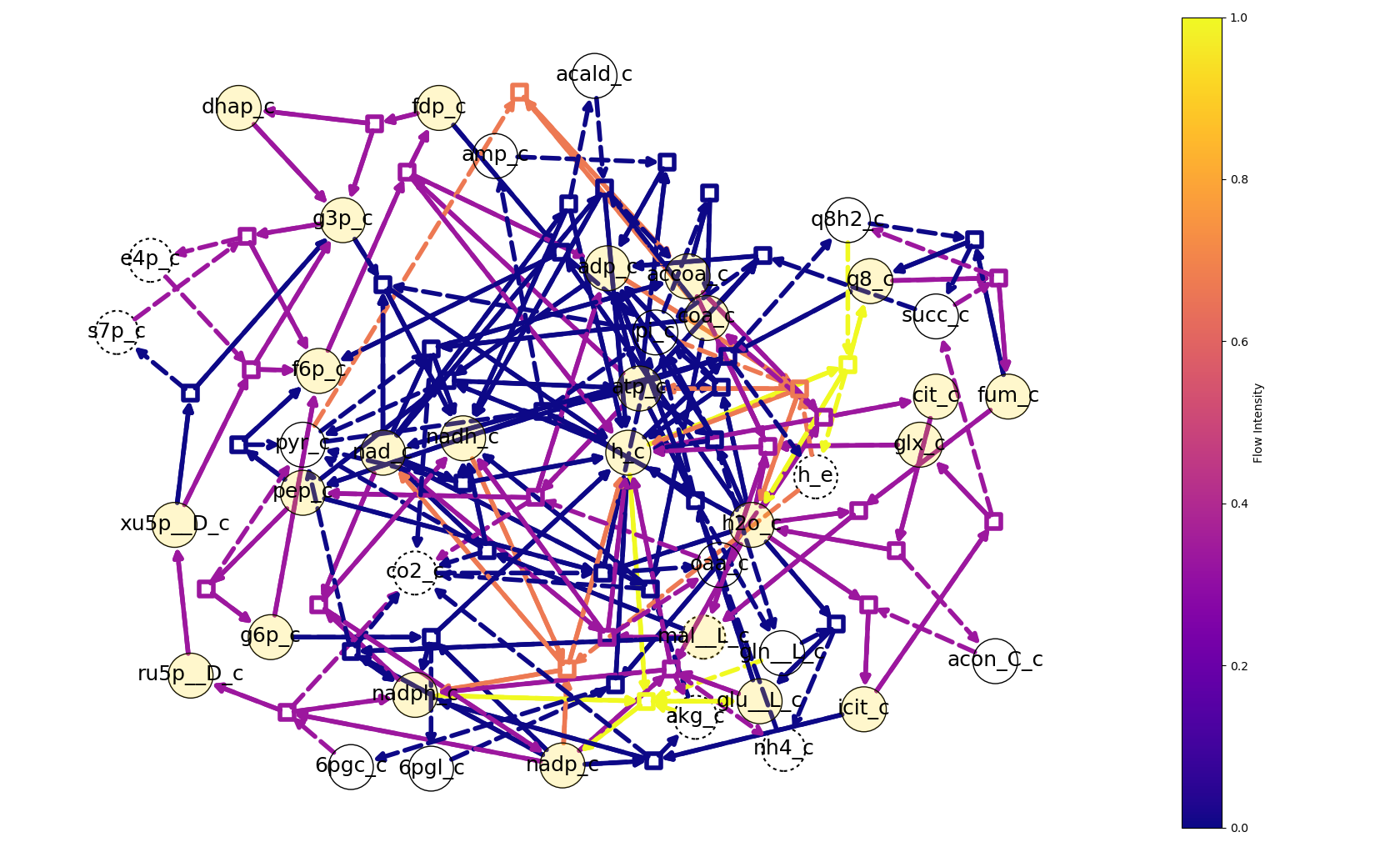}
    \includegraphics[width=0.46\textwidth]{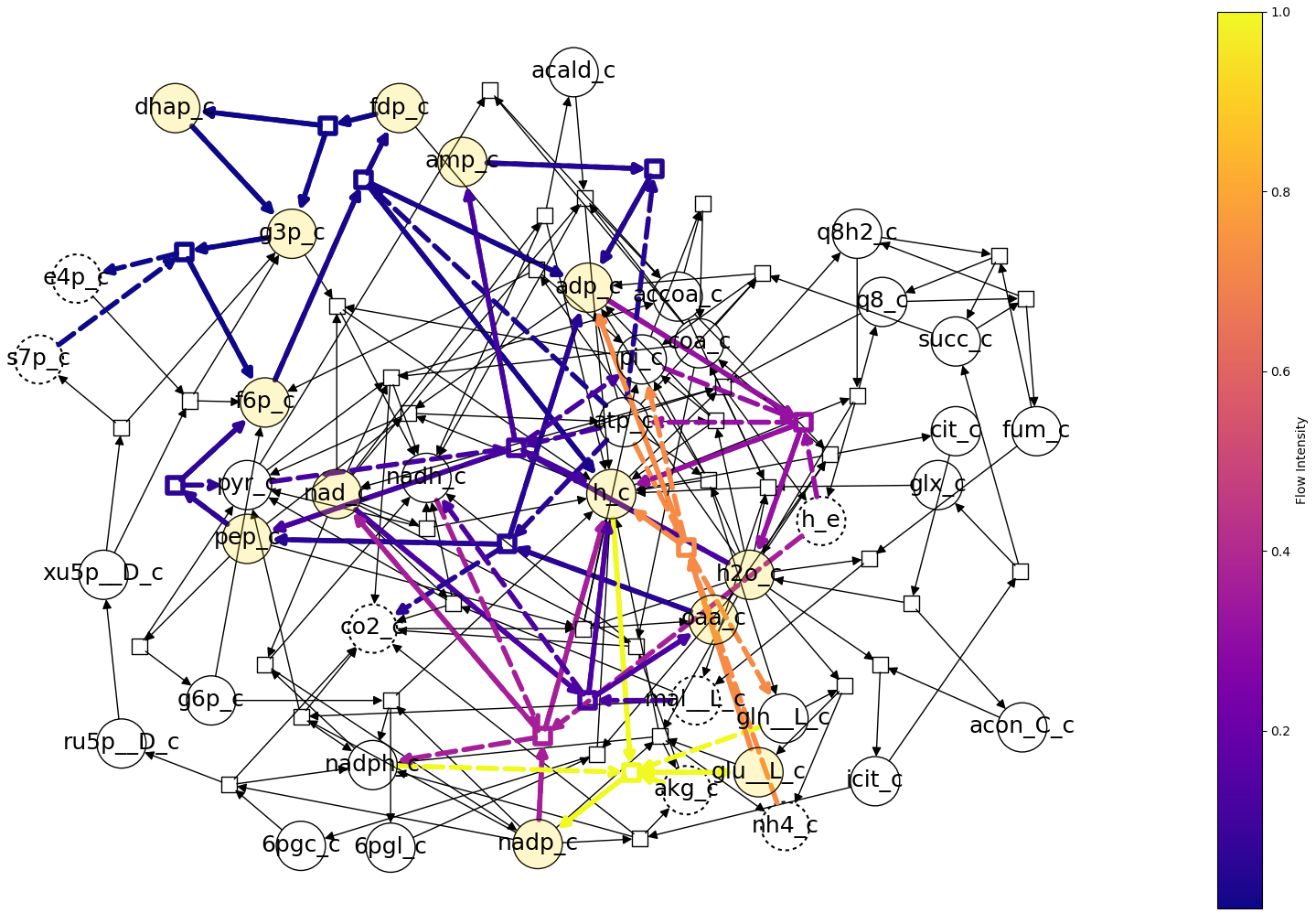}
    \caption{Application of Algorithms~\ref{alg:1} and~\ref{alg:3} to the \emph{E. coli} metabolic network. The left panel shows the largest self-sufficient subnetwork identified by Algorithm~\ref{alg:1}, while the right panel displays the autocatalytic subnetwork with the highest MAF (\textit{strongest subnetwork}) identified by Algorithm~\ref{alg:3} (also see Figure~\ref{fig:ecoli_zoom_subnet}). The legend indicates the optimal flow-intensity assigned to each reaction within the subnetwork.}
    \label{fig:ecoli_results_1}
\end{figure}

\begin{figure*}[t]
\centering
~\hspace*{-1.2cm}
\begin{subfigure}[b]{0.5\linewidth}
   \includegraphics[width=\textwidth]{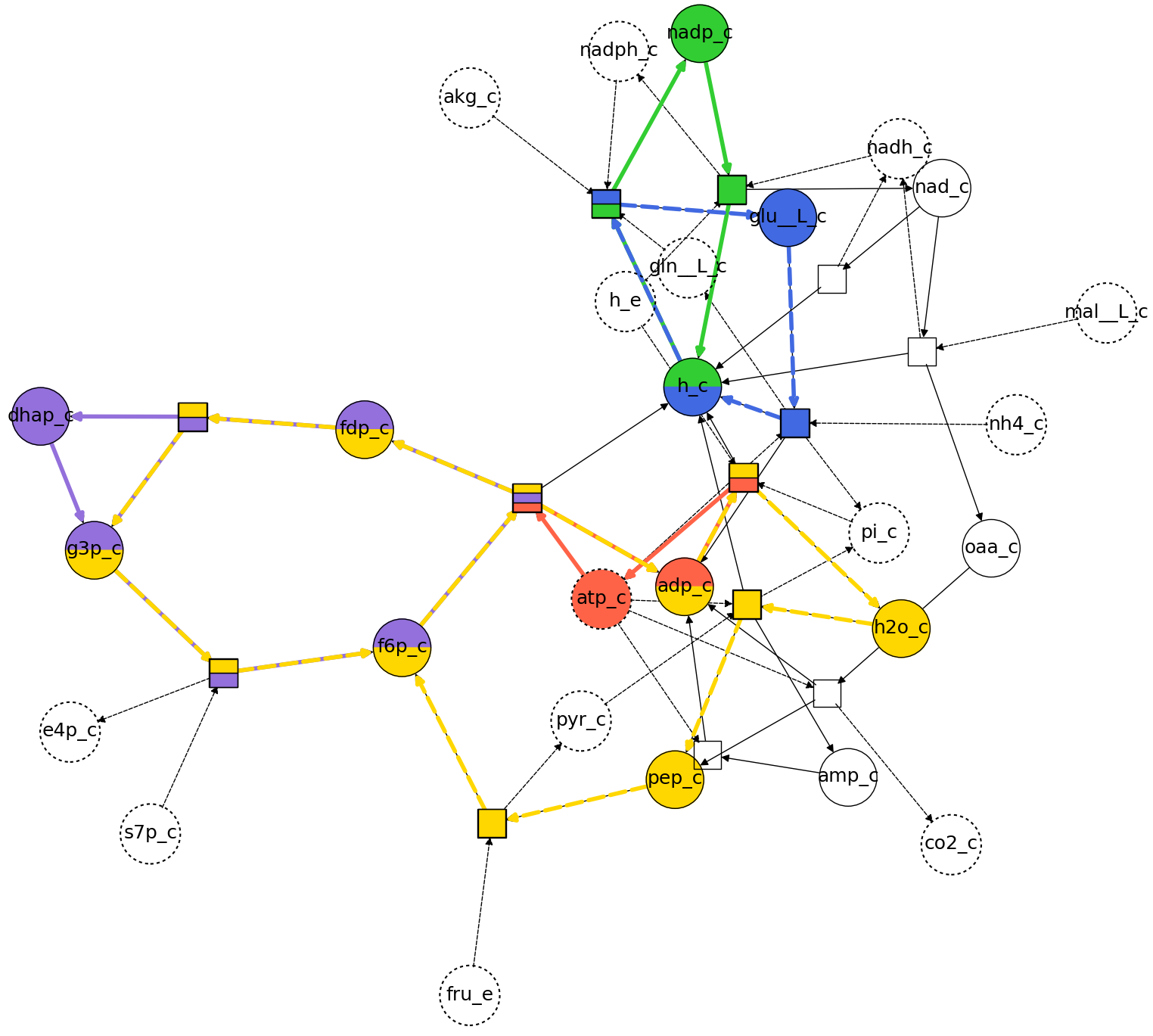}
\end{subfigure}
\begin{subfigure}[b]{0.35\textwidth}
   \includegraphics[width=\textwidth]{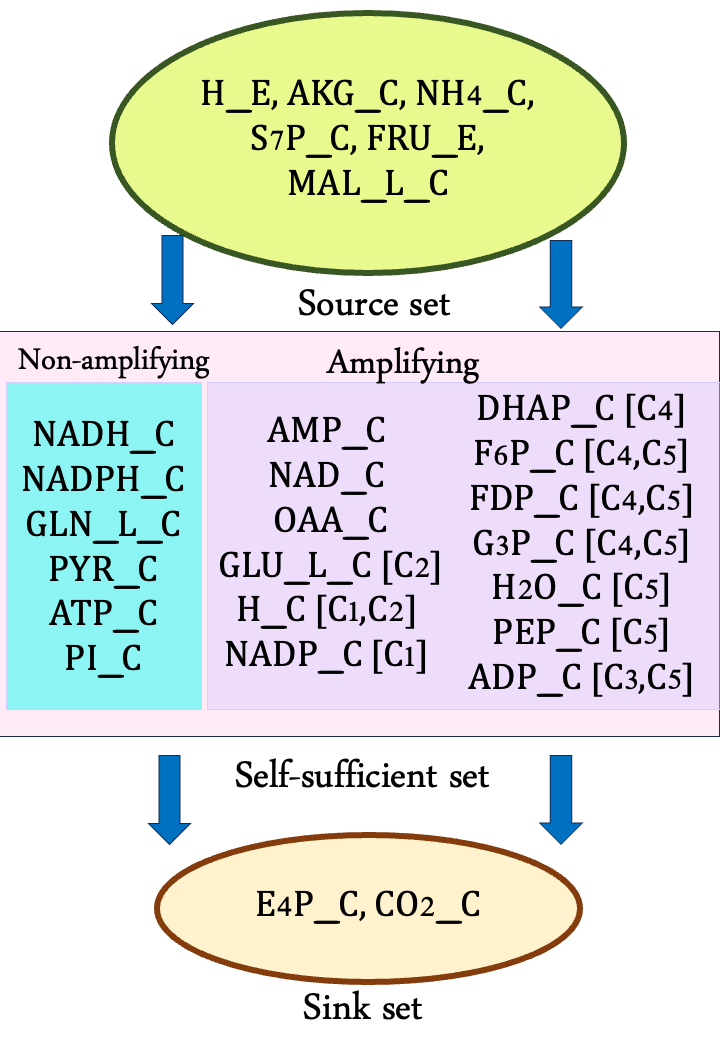}
\end{subfigure}
\caption[]{The cores (obtained by applying the algorithms in \cite{gagrani2023geometry}) contained in the solution of Algorithm \ref{alg:3} on the E.\ coli metabolism are shown in different colors in the left panel. Following Section~\ref{sec:features}, the right panel shows the species decomposition of the solution.
}
\label{fig:ecoli_zoom_subnet}
\end{figure*}

\begin{figure}[h!]
    \centering
    \includegraphics[width=0.38\textwidth]{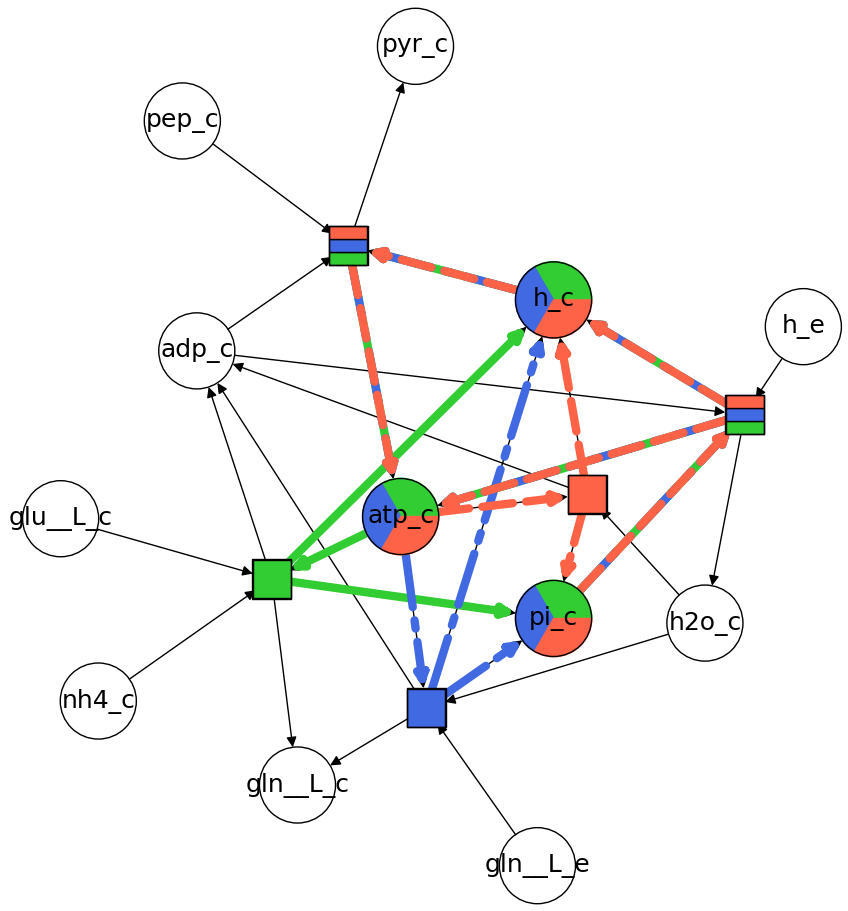}
    \includegraphics[width=0.61\textwidth]{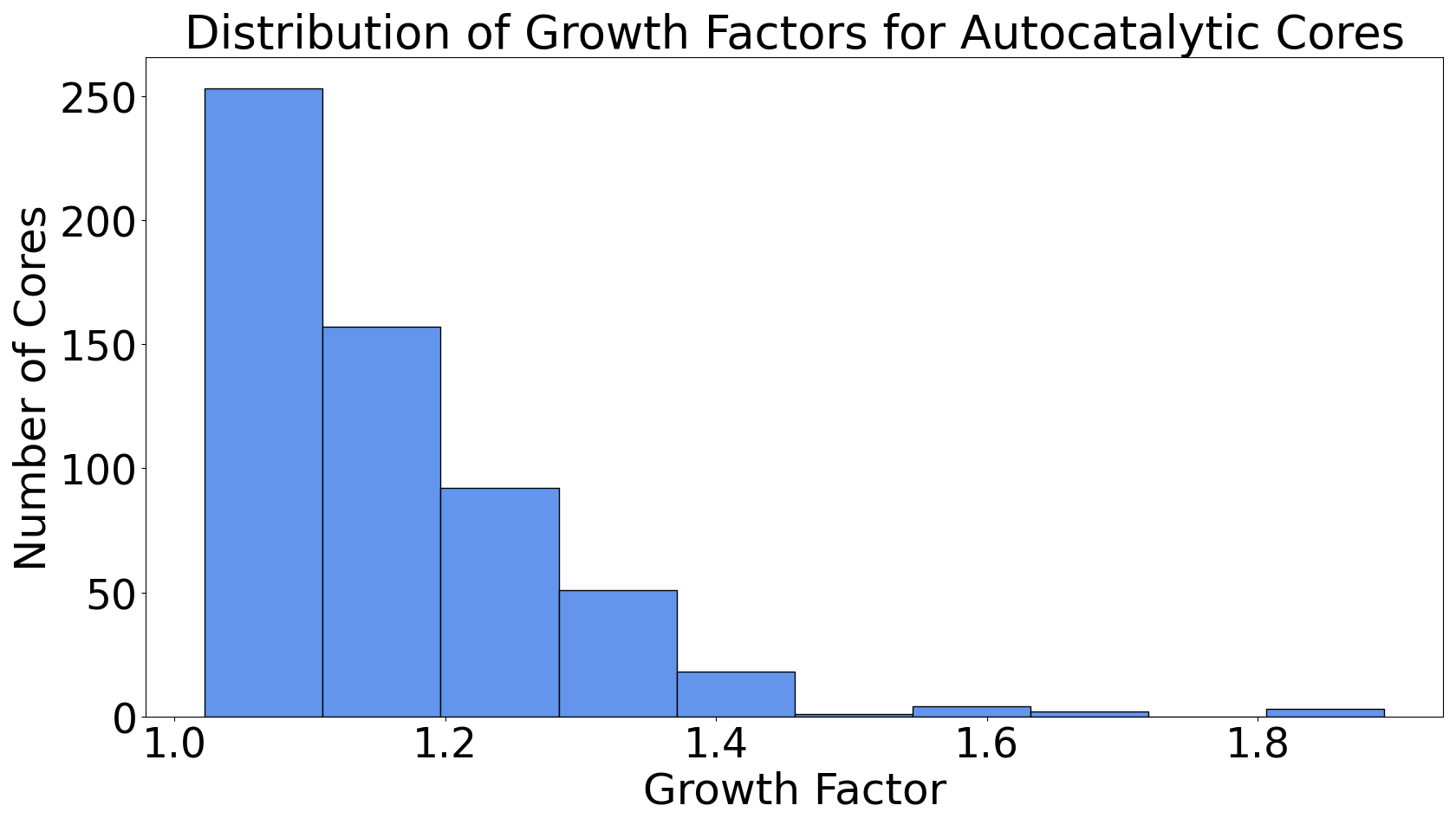}
    \caption{Autocatalytic cores in the \emph{E. coli} network. The left panel shows the top three cores, all with a MAF of $1.89$. Details of these cores are provided in Appendix~\ref{app:tables}. The right panel presents a histogram of the MAF values for all 581 autocatalytic cores identified in the network.}
    \label{fig:ecoli_results_2}
\end{figure}

\begin{figure}[h!]
    \centering
    \includegraphics[width=1\textwidth]{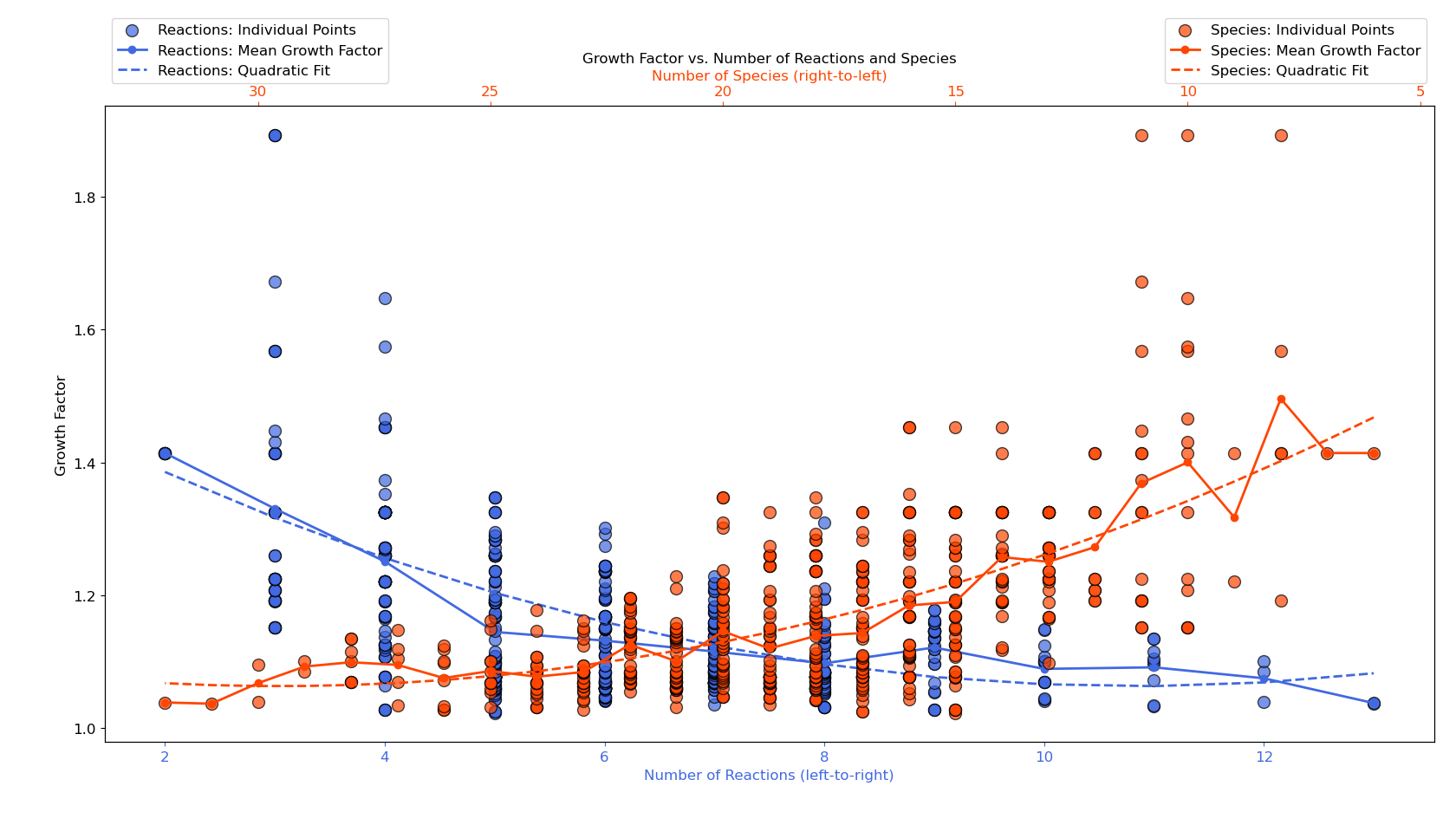}
    \caption{Relationship between MAF and the structural characteristics of autocatalytic cores in the \emph{E. coli} network, based on the number of species and reactions.}
    \label{fig:ecoli_results_3}
\end{figure}

\subsubsection*{Summary of Findings and Comparative Insights}
In the Formose network, both exhaustive enumeration and our algorithms identify the same 7‑species autocatalytic core as having the highest amplification (MAF=1.138), showing that extra reactions weaken self‑replication. In the more complex E. coli metabolism, individual cores reach MAF=1.89, but Algorithm~\ref{alg:3} finds a 13‑reaction subnetwork with MAF=2.77, exceeding any single core and revealing cooperative amplification. Interestingly, the largest autonomous subnetwork (25 species) is not autocatalytic, highlighting the balance between mass conservation and growth. Overall, Algorithms ~\ref{alg:1} and ~\ref{alg:3} both recover minimal autocatalytic motifs and expose higher‑order subnetworks that maximize self‑amplification.

\section{Conclusions}\label{sec:conclusions}

In this work, we introduce a novel framework for identifying and analyzing self-amplifying substructures in multi-directed hypergraphs, a generalization motivated by complex systems where multi-way interactions play a central role. We formalized the concept of self-amplification through the lens of a newly proposed index—the Maximal Amplification Factor (MAF)—which quantifies the net self-reinforcing potential of a subhypergraph under optimal intensity vectors on its hyperarcs.

To compute the MAF, we formulate a generalized  fractional optimization model and develop an exact iterative algorithm with convergence guarantees. We then extend this analysis to the combinatorial task of constructing subhypergraphs with the highest amplification factor, yielding a Mixed-Integer Nonlinear Programming (MINLP) model capable of identifying key subsystems that drive propagation dynamics. Our formulation accommodates real-world interpretability through additional structural constraints, making it adaptable to domain-specific applications such as chemistry, manufacturing, and social influence networks.

We validated our approach through two layers of computational study. First, experiments on synthetic hypergraphs simulating chemical reaction networks demonstrated the feasibility, robustness, and scalability of our methods, showcasing their ability to recover known amplification patterns and quantify emergent dynamics. Second, we conducted a detailed case study on two real-world chemical systems: the Formose reaction network and the \emph{E. coli} core metabolism. In both cases, our algorithms identified autocatalytic subnetworks with strong amplification signatures. For the Formose network, we recovered minimal subnetworks or cores and confirmed that the most compact structures exhibited the highest MAFs. In the more complex \emph{E. coli} network, our optimization uncovered a high-MAF subnetwork that outperformed individual cores, revealing emergent synergy among reactions.

These findings underscore the practical utility and generalizability of MAF-based analysis in uncovering catalytic and self-sustaining structures in both abstract and empirical settings. The combination of synthetic and real-data validation highlights the promise of our framework for studying self-reinforcement in higher-order networks.

This paper lays a solid foundation for computational studies in systems chemistry and prebiotic evolution, helping to elucidate fundamental principles that may have governed the origin of metabolic organization and life itself.

Overall, our study contributes to the growing body of work that seeks to extend combinatorial optimization into the realm of higher-order network structures. The notion of self-amplifying subhypergraphs offers a unifying abstraction for a wide array of processes involving self-sustaining or catalytic dynamics. Future directions include improving scalability through decomposition techniques, extending the framework to stochastic or time-evolving hypergraphs, and exploring domain-specific applications in greater depth, particularly in systems biology, process design, and network epidemiology.

\section*{Acknowledgements}

VB and GG were partially supported by grant PID2020-114594GB-C21 funded by MICIU/AEI/10.13039/501100011033;  grant RED2022-134149-T funded by MICIU/AEI
/10.13039/501100011033 (Thematic Network on Location Science and Related Problems); grant C-EXP-139-UGR23 funded by the Consejería de Universidad,
Investigación e Innovación and by the ERDF Andalusia Program 2021-2027, grant AT 21\_00032, and the IMAG-María de Maeztu grant CEX2020-001105-M /AEI /10.13039/501100011033. PG was partially funded by the National Science Foundation, Division of Environmental Biology (Grant No: DEB-2218817). PG was partially funded by the National Science Foundation, Division of Environmental Biology (Grant No: DEB-2218817).

 %\bibliographystyle{elsarticle-harv} 
 %\bibliographystyle{elsarticle-harv} 
% \bibliographystyle{cas-model2-names}
%  \bibliography{bib_growthfactor}

\clearpage
\appendix

\section{Detailed results for the case studies} \label{app:tables}

This section provides detailed results for the case studies on the Formose and E. coli datasets. For Formose, reaction subnetworks corresponding to different figures are shown with reaction numbers matching those in the dataset. For E. coli, experimental details, including compound names and identifiers (listed in Table \ref{table:ecoli_data}), are provided, with \texttt{\_c} and \texttt{\_e} suffixes indicating cytoplasmic and extracellular molecules, respectively; reaction numbers are arbitrary and ignorable.

%\noindent
%Green core:
\begin{table}[h!]
\centering
\caption{Summary of the green core from Fig. \ref{fig:formose_results_2}.}
\begin{tabular}{ll}
\hline
\textbf{Category}            & \textbf{Details}                                                                \\ \hline
\textbf{Source Set}            & C1a formaldehyde                                                             \\ \hline
\textbf{Autocatalytic Set}   & C2a, C2b, C3a, C3b, C4a, C4b, C5c                                           \\ \hline
\textbf{Amplification Factor}       & 1.138163596847235                                                               \\ \hline
\textbf{Reactions}           &                                                                                  \\ \hline
R1                           & C2a $\to$ C2b                                                                     \\
R2                           & C1a formaldehyde + C2b $\to$ C3a                                                 \\
R3                           & C3a $\to$ C3b                                                                     \\
R5                           & C1a formaldehyde + C3b $\to$ C4b                                                 \\
R6                           & C4b $\to$ C4a                                                                     \\
R9                           & C1a formaldehyde + C4a $\to$ C5c                                                 \\
R38                          & C5c $\to$ C2a + C3b                                                              \\ \hline
\end{tabular}
\label{tab:green_core}
\end{table}

%Blue core:
\begin{table}[h!]
\centering
\caption{Summary of the blue core from Fig. \ref{fig:formose_results_2}.}
\begin{tabular}{ll}
\hline
\textbf{Category}            & \textbf{Details}                                                                \\ \hline
\textbf{Source Set}            & C1a formaldehyde                                                             \\ \hline
\textbf{Autocatalytic Set}   & C3b, C3c dihydroxy acetone, C4b, C4c, C5d, C5e, C6e                          \\ \hline
\textbf{Amplification Factor}       & 1.1127643944856531                                                              \\ \hline
\textbf{Reactions}           &                                                                                  \\ \hline
R4                           & C3c dihydroxy acetone $\to$ C3b                                                  \\
R5                           & C1a formaldehyde + C3b $\to$ C4b                                                 \\
R7                           & C4b $\to$ C4c                                                                    \\
R10                          & C1a formaldehyde + C4c $\to$ C5d                                                 \\
R12                          & C5d $\to$ C5e                                                                    \\
R16                          & C1a formaldehyde + C5e $\to$ C6e                                                 \\
R37                          & C6e $\to$ C3b + C3c dihydroxy acetone                                            \\ \hline
\end{tabular}
\label{tab:blue_core}
\end{table}

%Red core:
\begin{table}[h!]
\centering
\caption{Summary of the red core from Fig. \ref{fig:formose_results_2}.}
\begin{tabular}{ll}
\hline
\textbf{Category}            & \textbf{Details}                                                                \\ \hline
\textbf{Source Set}            & C1a formaldehyde                                                             \\ \hline
\textbf{Autocatalytic Set}   & C2a, C2b, C3a, C3b, C4a, C4b, C5a, C5b, C6a, C6b, C7b                        \\ \hline
\textbf{Amplification Factor}       & 1.1056824557955995                                                              \\ \hline
\textbf{Reactions}           &                                                                                  \\ \hline
R1                           & C2a $\to$ C2b                                                                     \\
R2                           & C1a formaldehyde + C2b $\to$ C3a                                                 \\
R3                           & C3a $\to$ C3b                                                                     \\
R5                           & C1a formaldehyde + C3b $\to$ C4b                                                 \\
R6                           & C4b $\to$ C4a                                                                     \\
R8                           & C1a formaldehyde + C4a $\to$ C5b                                                 \\
R11                          & C5b $\to$ C5a                                                                     \\
R13                          & C1a formaldehyde + C5a $\to$ C6a                                                 \\
R17                          & C6a $\to$ C6b                                                                     \\
R20                          & C1a formaldehyde + C6b $\to$ C7b                                                 \\
R34                          & C7b $\to$ C2a + C5a                                                              \\ \hline
\end{tabular}
\label{tab:red_core}
\end{table}

%\noindent
%Green core:
\begin{table}[h!]
\centering
\caption{Summary of the green core from Fig. \ref{fig:ecoli_results_2}.}
\begin{tabular}{ll}
\hline
\textbf{Category}            & \textbf{Details}                                                                \\ \hline
\textbf{Source Set}            & pep\_c, h\_e, glu\_\_L\_c, nh4\_c                                           \\ \hline
\textbf{Sink Set}           & pyr\_c, h2o\_c, gln\_\_L\_c                                                  \\ \hline
\textbf{Non-autocatalytic Set} & adp\_c                                                                      \\ \hline
\textbf{Autocatalytic Set}   & atp\_c, h\_c, pi\_c                                                         \\ \hline
\textbf{Amplification Factor}       & 1.8932605011330286                                                              \\ \hline
\textbf{Reactions}           &                                                                                  \\ \hline
R18                          & adp\_c + pi\_c + 4h\_e $\to$ atp\_c + 3h\_c + h2o\_c                           \\
R20                          & adp\_c + h\_c + pep\_c $\to$ atp\_c + pyr\_c                                    \\
R41                          & atp\_c + glu\_\_L\_c + nh4\_c $\to$ adp\_c + h\_c + pi\_c + gln\_\_L\_c        \\ \hline
\end{tabular}
\label{tab:green_core_details}
\end{table}

%Blue core:
\begin{table}[h!]
\centering
\caption{Summary of the blue core from Fig. \ref{fig:ecoli_results_2}.}
\begin{tabular}{ll}
\hline
\textbf{Category}            & \textbf{Details}                                                                \\ \hline
\textbf{Source Set}            & pep\_c, h\_e, gln\_\_L\_e                                                   \\ \hline
\textbf{Sink Set}           & pyr\_c, gln\_\_L\_c                                                         \\ \hline
\textbf{Non-autocatalytic Set} & adp\_c, h2o\_c                                                              \\ \hline
\textbf{Autocatalytic Set}   & atp\_c, h\_c, pi\_c                                                         \\ \hline
\textbf{Amplification Factor}       & 1.8932605011330286                                                              \\ \hline
\textbf{Reactions}           &                                                                                  \\ \hline
R18                          & adp\_c + pi\_c + 4h\_e $\to$ atp\_c + 3h\_c + h2o\_c                           \\
R20                          & adp\_c + h\_c + pep\_c $\to$ atp\_c + pyr\_c                                    \\
R42                          & atp\_c + h2o\_c + gln\_\_L\_e $\to$ adp\_c + h\_c + pi\_c + gln\_\_L\_c        \\ \hline
\end{tabular}
\label{tab:blue_core_details}
\end{table}

%Red core:
\begin{table}[h!]
\centering
\caption{Summary of the red core from Fig. \ref{fig:ecoli_results_2}.}
\begin{tabular}{ll}
\hline
\textbf{Category}            & \textbf{Details}                                                                \\ \hline
\textbf{Source Set}            & pep\_c, h\_e                                                                \\ \hline
\textbf{Sink Set}           & pyr\_c                                                                      \\ \hline
\textbf{Non-autocatalytic Set} & adp\_c, h2o\_c                                                              \\ \hline
\textbf{Autocatalytic Set}   & atp\_c, h\_c, pi\_c                                                         \\ \hline
\textbf{Amplification Factor}       & 1.8932605011330286                                                              \\ \hline
\textbf{Reactions}           &                                                                                  \\ \hline
R13                          & atp\_c + h2o\_c $\to$ adp\_c + h\_c + pi\_c                                     \\
R18                          & adp\_c + pi\_c + 4h\_e $\to$ atp\_c + 3h\_c + h2o\_c                           \\
R20                          & adp\_c + h\_c + pep\_c $\to$ atp\_c + pyr\_c                                    \\ \hline
\end{tabular}
\label{tab:red_core_details}
\end{table}

\begin{table}[h!]
\centering
\caption{Summary of data from Algorithm \ref{alg:3} on the E. coli Dataset.}
\begin{tabular}{ll}
\hline
\textbf{Category}           & \textbf{Details}                                                                                       \\ \hline
\textbf{Amplification Factor}      & 2.7678080214824665                                                                                     \\ \hline
\textbf{Reactions}          &                                                                                                        \\ \hline
                          & atp\_c + f6p\_c $\to$ adp\_c + fdp\_c + h\_c                                                          \\
                         & atp\_c + oaa\_c $\to$ adp\_c + co2\_c + pep\_c                                                       \\
                         & atp\_c + pyr\_c + h2o\_c $\to$ 2h\_c + pep\_c + pi\_c + amp\_c                                      \\
                         & atp\_c + amp\_c $\to$ 2adp\_c                                                                          \\
                         & adp\_c + pi\_c + 4h\_e $\to$ atp\_c + 3h\_c + h2o\_c                                                \\
                        & g3p\_c + s7p\_c $\to$ f6p\_c + e4p\_c                                                                 \\
                        & nadh\_c + 2h\_e + nadp\_c $\to$ 2h\_c + nad\_c + nadph\_c                                          \\
                        & dhap\_c $\to$ g3p\_c                                                                                  \\
                        & fdp\_c $\to$ g3p\_c + dhap\_c                                                                         \\
                        & pep\_c + fru\_e $\to$ f6p\_c + pyr\_c                                                                 \\
                        & atp\_c + glu\_\_L\_c + nh4\_c $\to$ adp\_c + h\_c + pi\_c + gln\_\_L\_c                          \\
                      & h\_c + akg\_c + gln\_\_L\_c + nadph\_c $\to$ 2glu\_\_L\_c + nadp\_c                               \\
                        & nad\_c + mal\_\_L\_c $\to$ h\_c + nadh\_c + oaa\_c                                                   \\ \hline
\end{tabular}
\label{tab:ecoli_alg3_results}
\end{table}

% ------------------------------------------------------------------------------------
\begin{table}[h!]
\centering
\caption{Summary of Core 1 (Green) from Fig. \ref{fig:ecoli_zoom_subnet}.}
\begin{tabular}{ll}
\hline
\textbf{Category}            & \textbf{Details}                                                                \\ \hline
\textbf{Source Set}            & nadh\_c, akg\_c, h\_e, gln\_\_L\_c                                           \\ \hline
\textbf{Sink Set}           & nad\_c, glu\_\_L\_c                                                         \\ \hline
\textbf{Non-autocatalytic Set} & nadph\_c                                                                   \\ \hline
\textbf{Autocatalytic Set}   & h\_c, nadp\_c                                                               \\ \hline
\textbf{Growth Factor}       & 1.4142011834319526                                                              \\ \hline
\textbf{Reactions}           &                                                                                  \\ \hline
R7                           & nadh\_c + 2h\_e + nadp\_c $\to$ 2h\_c + nad\_c + nadph\_c                       \\
R12                          & h\_c + akg\_c + gln\_\_L\_c + nadph\_c $\to$ 2glu\_\_L\_c + nadp\_c             \\ \hline
\end{tabular}
\label{tab:green_core_algorithm3}
\end{table}
% ----
\begin{table}[h!]
\centering
\caption{Summary of Core 2 (Blue) from Fig. \ref{fig:ecoli_zoom_subnet}.}
\begin{tabular}{ll}
\hline
\textbf{Category}            & \textbf{Details}                                                                \\ \hline
\textbf{Source Set}            & atp\_c, akg\_c, nadph\_c, nh4\_c                                            \\ \hline
\textbf{Sink Set}           & adp\_c, pi\_c, nadp\_c                                                      \\ \hline
\textbf{Non-autocatalytic Set} & gln\_\_L\_c                                                                 \\ \hline
\textbf{Autocatalytic Set}   & h\_c, glu\_\_L\_c                                                           \\ \hline
\textbf{Growth Factor}       & 1.4142011834319526                                                              \\ \hline
\textbf{Reactions}           &                                                                                  \\ \hline
R11                          & atp\_c + glu\_\_L\_c + nh4\_c $\to$ adp\_c + h\_c + pi\_c + gln\_\_L\_c          \\
R12                          & h\_c + akg\_c + gln\_\_L\_c + nadph\_c $\to$ 2glu\_\_L\_c + nadp\_c             \\ \hline
\end{tabular}
\label{tab:blue_core_algorithm3}
\end{table}
% ----
\begin{table}[h!]
\centering
\caption{Summary of Core 3 (Orange) from Fig. \ref{fig:ecoli_zoom_subnet}.}
\begin{tabular}{ll}
\hline
\textbf{Category}            & \textbf{Details}                                                                \\ \hline
\textbf{Source Set}            & pi\_c, amp\_c, h\_e                                                         \\ \hline
\textbf{Sink Set}           & h\_c, h2o\_c                                                                \\ \hline
\textbf{Autocatalytic Set}   & adp\_c, atp\_c                                                              \\ \hline
\textbf{Growth Factor}       & 1.4142011834319526                                                              \\ \hline
\textbf{Reactions}           &                                                                                  \\ \hline
R4                           & atp\_c + amp\_c $\to$ 2adp\_c                                                   \\
R5                           & adp\_c + pi\_c + 4h\_e $\to$ atp\_c + 3h\_c + h2o\_c                           \\ \hline
\end{tabular}
\label{tab:orange_core_algorithm3}
\end{table}
% ----
\begin{table}[h!]
\centering
\caption{Summary of Core 4 (Purple) from Fig. \ref{fig:ecoli_zoom_subnet}.}
\begin{tabular}{ll}
\hline
\textbf{Category}            & \textbf{Details}                                                                \\ \hline
\textbf{Source Set}            & atp\_c, s7p\_c                                                              \\ \hline
\textbf{Sink Set}           & adp\_c, h\_c, e4p\_c                                                        \\ \hline
\textbf{Autocatalytic Set}   & f6p\_c, fdp\_c, g3p\_c, dhap\_c                                             \\ \hline
\textbf{Growth Factor}       & 1.2207340620350884                                                              \\ \hline
\textbf{Reactions}           &                                                                                  \\ \hline
R1                           & atp\_c + f6p\_c $\to$ adp\_c + fdp\_c + h\_c                                    \\
R6                           & g3p\_c + s7p\_c $\to$ f6p\_c + e4p\_c                                           \\
R8                           & dhap\_c $\to$ g3p\_c                                                            \\
R9                           & fdp\_c $\to$ g3p\_c + dhap\_c                                                  \\ \hline
\end{tabular}
\label{tab:purple_core_algorithm3}
\end{table}
% ----
\begin{table}[h!]
\centering
\caption{Summary of Core 5 (Yellow) from Fig. \ref{fig:ecoli_zoom_subnet}.}
\begin{tabular}{ll}
\hline
\textbf{Category}            & \textbf{Details}                                                                \\ \hline
\textbf{Source Set}            & h\_e, s7p\_c, fru\_e                                                        \\ \hline
\textbf{Sink Set}           & h\_c, amp\_c, e4p\_c, dhap\_c                                               \\ \hline
\textbf{Non-autocatalytic Set} & atp\_c, pyr\_c, pi\_c                                                     \\ \hline
\textbf{Autocatalytic Set}   & adp\_c, f6p\_c, fdp\_c, h2o\_c, pep\_c, g3p\_c                            \\ \hline
\textbf{Growth Factor}       & 1.2207340620350884                                                              \\ \hline
\textbf{Reactions}           &                                                                                  \\ \hline
R1                           & atp\_c + f6p\_c $\to$ adp\_c + fdp\_c + h\_c                                    \\
R3                           & atp\_c + pyr\_c + h2o\_c $\to$ 2h\_c + pep\_c + pi\_c + amp\_c                 \\
R5                           & adp\_c + pi\_c + 4h\_e $\to$ atp\_c + 3h\_c + h2o\_c                           \\
R6                           & g3p\_c + s7p\_c $\to$ f6p\_c + e4p\_c                                           \\
R9                           & fdp\_c $\to$ g3p\_c + dhap\_c                                                  \\
R10                          & pep\_c + fru\_e $\to$ f6p\_c + pyr\_c                                          \\ \hline
\end{tabular}
\label{tab:yellow_core_algorithm3}
\end{table}
% ------------------------------------------------------------------------------------

\begin{table}[]
\scriptsize
\center
\begin{tabular}{cc}
\textbf{Identifier}   & \textbf{Compound}                                              \\
\hline
adp\_c       & ADP C10H12N5O10P2                                     \\
atp\_c       & ATP C10H12N5O13P3                                     \\
f6p\_c       & D-Fructose 6-phosphate                                \\
fdp\_c       & D-Fructose 1                                          \\
h\_c         & H+                                                    \\
accoa\_c     & Acetyl-CoA                                            \\
coa\_c       & Coenzyme A                                            \\
for\_c       & Formate                                               \\
pyr\_c       & Pyruvate                                              \\
g6p\_c       & D-Glucose 6-phosphate                                 \\
13dpg\_c     & 3-Phospho-D-glyceroyl phosphate                       \\
3pg\_c       & 3-Phospho-D-glycerate                                 \\
6pgc\_c      & 6-Phospho-D-gluconate                                 \\
6pgl\_c      & 6-phospho-D-glucono-1                                 \\
h2o\_c       & H2O H2O                                               \\
acald\_c     & Acetaldehyde                                          \\
nad\_c       & Nicotinamide adenine dinucleotide                     \\
nadh\_c      & Nicotinamide adenine dinucleotide - reduced           \\
2pg\_c       & D-Glycerate 2-phosphate                               \\
etoh\_c      & Ethanol                                               \\
ac\_c        & Acetate                                               \\
actp\_c      & Acetyl phosphate                                      \\
co2\_c       & CO2 CO2                                               \\
oaa\_c       & Oxaloacetate                                          \\
pep\_c       & Phosphoenolpyruvate                                   \\
pi\_c        & Phosphate                                             \\
acon\_C\_c   & Cis-Aconitate                                         \\
cit\_c       & Citrate                                               \\
icit\_c      & Isocitrate                                            \\
amp\_c       & AMP C10H12N5O7P                                       \\
akg\_c       & 2-Oxoglutarate                                        \\
succoa\_c    & Succinyl-CoA                                          \\
h\_e         & H+                                                    \\
e4p\_c       & D-Erythrose 4-phosphate                               \\
g3p\_c       & Glyceraldehyde 3-phosphate                            \\
gln\_\_L\_c  & L-Glutamine                                           \\
glu\_\_L\_c  & L-Glutamate                                           \\
nadp\_c      & Nicotinamide adenine dinucleotide phosphate           \\
nadph\_c     & Nicotinamide adenine dinucleotide phosphate - reduced \\
r5p\_c       & Alpha-D-Ribose 5-phosphate                            \\
ru5p\_\_D\_c & D-Ribulose 5-phosphate                                \\
xu5p\_\_D\_c & D-Xylulose 5-phosphate                                \\
o2\_c        & O2 O2                                                 \\
q8\_c        & Ubiquinone-8                                          \\
q8h2\_c      & Ubiquinol-8                                           \\
fum\_c       & Fumarate                                              \\
succ\_c      & Succinate                                             \\
s7p\_c       & Sedoheptulose 7-phosphate                             \\
dhap\_c      & Dihydroxyacetone phosphate                            \\
fru\_e       & D-Fructose                                            \\
mal\_\_L\_c  & L-Malate                                              \\
glc\_\_D\_e  & D-Glucose                                             \\
nh4\_c       & Ammonium                                              \\
gln\_\_L\_e  & L-Glutamine                                           \\
glx\_c       & Glyoxylate                                            \\
lac\_\_D\_c  & D-Lactate                                            
\end{tabular}
\caption{Table of identifier and compound names for the E.\ coli dataset.}\label{table:ecoli_data}
\end{table}

\end{document}